\documentclass[12pt]{amsart}
\usepackage{amsmath,amssymb,amscd,amsthm,graphicx,enumerate}

\usepackage{hyperref}
\hypersetup{breaklinks=true}

\numberwithin{equation}{section}
\theoremstyle{plain}



\newcommand{\A}{\mathbb{A}}
\newcommand{\B}{\mathbb{B}}
\DeclareMathOperator{\Tr}{Tr}
\DeclareMathOperator{\Str}{Str}
\DeclareMathOperator{\Ker}{Ker}
\DeclareMathOperator{\Coker}{Coker}
\DeclareMathOperator{\Dom}{Dom}
\DeclareMathOperator{\Ch}{Ch}

\DeclareMathOperator{\id}{id}

\newcommand{\C}{{\mathbb C}}

\newcommand{\Diff}{\operatorname{Diff}}

\newcommand{\End}{\operatorname{End}}

\newcommand{\E}{\mathbb{E}}

\newcommand{\F}{\mathbb{F}}

\newcommand{\HH}{\operatorname{H}}

\newcommand{\Id}{\operatorname{Id}}

\newcommand{\Image}{\operatorname{Im}}

\newcommand{\Ind}{\operatorname{Ind}}

\newcommand{\K}{\operatorname{K}}

\newcommand{\N}{{\mathbb N}}

\newcommand{\R}{{\mathbb R}}

\newcommand{\sign}{\operatorname{sign}}

\newcommand{\supp}{\operatorname{supp}}

\newcommand{\Todd}{\operatorname{Todd}}

\newcommand{\Z}{{\mathbb Z}}

\newcommand{\tH}{\widetilde{\mathcal{H}}}
\newcommand{\ta}{\widetilde{\mathbb{A}}}

\newcommand{\tg}{\widetilde{\gamma}}

\newtheorem{claim}{Claim}
\newtheorem{lemma}{Lemma}
\newtheorem{theorem}{Theorem}
\newtheorem{proposition}{Proposition}
\newtheorem{corollary}{Corollary}

\theoremstyle{definition}

\newtheorem{remark}{Remark}
\newtheorem{example}{Example}
\newtheorem{definition}{Definition}
\newtheorem{fact}{Fact}
\begin{document}

\begin{abstract}

We give an infinite dimensional
description of the differential $K$-theory of a manifold $M$.
The generators are
triples $\left[ {\mathcal H}, \A, \omega \right]$ where
${\mathcal H}$ is a $\Z_2$-graded Hilbert bundle  on $M$,
$\A$ is a superconnection on ${\mathcal H}$ and $\omega$ is a
differential form on $M$. The relations involve eta forms.
We show that the ensuing group is the differential $K$-group
$\check{K}^0(M)$.
In addition, we construct the pushforward of a finite dimensional
cocycle under a proper submersion with a Riemannian structure.
We give the analogous description of the odd differential $K$-group
$\check{K}^1(M)$. 
Finally, we give a model for
twisted differential $K$-theory.
\end{abstract}

\title{A Hilbert bundle description of differential $K$-theory}
\author{Alexander Gorokhovsky}
\address{Department of Mathematics\\
University of Colorado, Boulder\\
Campus Box 395 \\
Boulder, CO  80309-0395}
\email{gorokhov@euclid.colorado.edu}

\author{John Lott}
\address{Department of Mathematics\\
University of California at Berkeley \\
Berkeley, CA  94720-3840}
\email{lott@berkeley.edu}
\thanks{Research supported by NSF grants DMS-0900968 and
DMS-1510192, and a Simons Fellowship}
\date{January 17, 2018}
\maketitle

\section{Introduction}

Differential $K$-groups are invariants of smooth manifolds that
combine $K$-theory with differential forms.
As shown in \cite{Freed-Lott (2010)}, many results from
local index theory fit into the framework of differential
$K$-theory.  For background and history about differential
$K$-theory, we refer to the introduction of \cite{Freed-Lott (2010)}.

Just as $K$-theory has
different but equivalent descriptions, so does differential
$K$-theory. The primary goal of this paper is to give a new
description of differential $K$-theory, based on Hilbert bundles,
that unifies other descriptions.  A secondary goal is
to provide a functional analytic framework for superconnections
on Hilbert bundles.

We recall the generators for some of the descriptions of the $K$-group
$K^0(M)$ of a compact manifold $M$ :
\begin{enumerate}
\item Vector bundles on $M$ \cite{Atiyah (1994)}.
\item Maps from $M$ to the space of Fredholm operators
\cite[Appendix]{Atiyah (1994)}.
\item Maps $p : Z \rightarrow M$ where $Z$ is compact and $p$ is
$K$-oriented
\cite{Connes-Skandalis (1984)}.
\item $\Z_2$-graded Hilbert $C(M)$-modules equipped with certain bounded operators that
commute with $C(M)$
\cite{Kasparov (1988),Mishchenko-Fomenko (1979)}.
\item $\Z_2$-graded Hilbert $C(M)$-modules equipped with certain
possibly-unbounded operators that
commute with $C(M)$
\cite{Baaj-Julg (1983)}.
\end{enumerate}
Of these descriptions, perhaps the last one,
based on unbounded $KK$-cycles, is the most encompassing one.

For the first three descriptions of $K$-theory,
there are corresponding models for differential $K$-theory :
\begin{enumerate}
\item Vector bundles with connections, as in
\cite{Freed-Lott (2010),Simons-Sullivan (2010)}.
\item The Hopkins-Singer model \cite{Hopkins-Singer (2005)}
\item The geometric families of Bunke-Schick \cite{Bunke-Schick (2009)}.
\end{enumerate}

All of these descriptions give isomorphic groups, which we denote by
$\check{K}^*_{stan}(M)$.
In this paper we give a new model for the differential $K$-theory of $M$,
extending the description of $K$-theory using unbounded $KK$-cycles.
Our model is in terms of
Hilbert bundles on $M$ equipped with superconnections.
One motivation for our model is that it unifies
earlier models, as described below.
The main result of the paper is the following.

\begin{theorem} \label{thm1}
The differential $K$-groups $\check{K}^*(M)$, as
defined using Hilbert bundles and superconnections, are
isomorphic to $\check{K}^*_{stan}(M)$.
\end{theorem}

Given a finite dimensional Hermitian vector bundle on $M$ with
compatible connection, we can think of it as a Hilbert bundle on
$M$ with a very special superconnection.  Hence our Hilbert bundle model
includes the standard description of differential $K$-theory using
finite dimensional vector bundles with connection.
Similarly, given a geometric
family in the sense of \cite[Section 2]{Bunke-Schick (2009)}, there is an
ensuing Hilbert bundle equipped with the Bismut superconnection
\cite{Bismut (1986)}.
Hence our model
also includes the description of differential $K$-theory using geometric
families.  However, we do not have an obvious way to construct
a Hilbert bundle, with superconnection, from a
Hopkins-Singer cocycle \cite[Section 4.4]{Hopkins-Singer (2005)}.

To motivate the use of superconnections, we recall that in the
vector bundle description of $\check{K}^0_{stan}(M)$, the generators
are triples $[E,\nabla, \omega]$, where $E$ is a finite
dimensional Hermitian $\Z_2$-graded
vector bundle on $M$, $\nabla$ is a compatible connection and
$\omega \in \Omega^{odd}(M)/\Image(d)$. The relations involve
Chern-Simons forms.  There is an equivalent description whose
generators are triples $[E, \A, \omega]$, where $\A$ is a
compatible superconnection on $E$ in the sense of Quillen
\cite{Quillen (1985)}, and whose relations involve eta forms.
When we pass to infinite dimensional vector bundles, the
Chern character construction using connections no longer make
sense.  However, under suitable hypotheses, we show that
the construction using superconnections does make sense.

Hence one goal of this paper is to find the right setting
for superconnections on Hilbert bundles. As an indication,
such a setting should allow for Bismut superconnections.
A first question is what the structure group $G$ of such a Hilbert
bundle should be.  The answer to this is not immediately evident.
One remark is that to deal with a geometric family
whose fiber is a compact manifold $Z$, there should be a
smooth homomorphism from $\Diff(Z)$ to $G$ coming from the action
of $\Diff(Z)$ on the Hilbert space $H$ of
square-integrable half-densities on $Z$.
However, with the norm topology on $U(H)$,
such an action is not even continuous (as seen in the
action on the circle on itself by rotations).

To construct $G$ in general, we use the data of a
Hilbert space $H$ and
an unbounded self-adjoint operator $D$ on $H$ that is $\theta$-summable
for all $\theta > 0$
(such as a Dirac-type operator). Using $D$,
in Section \ref{sect1} we define Sobolev
spaces $H^s$ and pseudodifferential operators $op^k$ that map $H^s$
to $H^{s-k}$,
following Connes and Moscovici
\cite[Appendix B]{Connes-Moscovici (1995)}.
As a set, we take $G = U(H) \cap op^0$. To put a smooth
structure on $G$, we can use the fact that we only care about
Hilbert bundles over finite dimensional manifolds, as opposed to
more general base spaces.  Hence it
suffices to say what a smooth map, from a domain in Euclidean space
to $G$, should be.  This is the underlying idea of diffeological
smooth structures \cite{Iglesias (2013)}. In our case,
we say that such maps are smooth if
they are compatible in a certain sense
with the Fr\'echet topologies on $H^s$ and $op^k$.

Given a Hilbert bundle ${\mathcal H}$ on $M$, with such a structure
group, in Section \ref{sect2}
we develop the theory of superconnections $\A$ on ${\mathcal H}$.
We construct their Chern characters and eta forms.
In Section \ref{sect3} we give our generators and relations for
$\check{K}^0(M)$.
The generators are triples $[{\mathcal H}, \A, \omega]$ where
${\mathcal H}$ is a $\Z_2$-graded Hilbert bundle on $M$,
$\A$ is a superconnection on ${\mathcal H}$ and
$\omega \in \Omega^{odd}(M)/\Image(d)$. There are three relations.
The first relation is about taking direct sums.
The second relation arises when the degree-$0$ part
$\A_{[0]}$ of the superconnection is invertible, and involves
an eta form.  The third relation says what happens when one
changes $\A_{[0]}$ by a family of operators in $op^0$, and
involves a relative eta form.

If $c$ is a generator for $\check{K}^0(M)$ then we construct a generator
$q(c)$ for $\check{K}^0_{stan}(M)$, based on certain choices.
We show that the class of $q(c)$ in $\check{K}^0_{stan}(M)$ is
independent of the choices.  We prove that $q$ passes to a map
$q : \check{K}^0(M) \rightarrow
\check{K}^0_{stan}(M)$. We then show that $q$ is an isomorphism, thereby
proving Theorem \ref{thm1}.

One advantage of a Hilbert bundle approach to
differential $K$-groups is that,
as in \cite{Bunke-Schick (2009)}, the pushforward becomes
essentially tautological.
We recall that using the finite dimensional model
$\check{K}^0_{stan}$ for
differential $K$-theory, in
\cite{Freed-Lott (2010)} 
two pushforwards were defined,
called the analytic index 
and the topological index.
The definition of the analytic index there involved some
perturbations.
The main theorem of \cite{Freed-Lott (2010)} said
that the two indices coincide. 
Given a fiber bundle $\pi : M \rightarrow B$
with even dimensional compact fibers and a Riemannian structure,
in Section \ref{sect4} we
define the pushforward $\pi_* [{\mathcal H}, \A, \omega]$
of a finite dimensional representative
$[{\mathcal H}, \A, \omega]$ for $\check{K}^0(M)$.

\begin{theorem} \label{newtheorem2}
The pushforward $\pi_*$ on cocycles passes to a map $\pi_* : \check{K}^0(M)
  \rightarrow \check{K}^0(B)$.
  It coincides with the
  analytic index of \cite{Freed-Lott (2010)}.
\end{theorem}

The proof of Theorem \ref{newtheorem2} partially uses results from
\cite{Bunke-Ma (2004),Liu (2015)}. 

Another advantage of using Hilbert bundles is that it allows a unified
treatment of even and odd differential $K$-groups.  (By way of
contrast, in \cite[Section 9]{Freed-Lott (2010)} and
\cite{Tradler-Wilson-Zeinalian (2013)}, the odd differential
$K$-groups
were constructed based on the model of odd $K$-theory coming from maps
to unitary groups.)
In Section \ref{sect5} we indicate how the results of the preceding
sections extend to the odd differential $K$-group $\check{K}^1(M)$.

Yet another advantage of using Hilbert bundles is that it allows for a
simple model of twisted differential $K$-theory.  We recall that ordinary
$K$-theory can be twisted by an element of $\HH^3(M; \Z)$.
The corresponding twisted differential $K$-theory has been considered in
many papers, including
\cite{Bunke-Nikolaus (2014),Bunke-Schick (2012),Carey-Mickelsson-Wang (2009),Park (2016)}.
Some applications to mathematical physics are in
\cite{Kahle-Valentino (2014)}.
In Section \ref{twisted} we give the basic definitions for a
Hilbert bundle model $\check{K}^0_{\mathcal L}(M)$ of
twisted differential $K$-theory, where ${\mathcal L}$ is a
unitary gerbe on $M$. 
If one restricts to finite dimensional vector bundles in the definition
then one can only deal with twistings by torsion elements of $\HH^3(M; \Z)$.

There are many directions for further study, including the following.\\
1. Extend the pushforward to infinite dimensional cocycles for
$\check{K}^0$; see Remark \ref{added}.\\ 
2. Show that the twisted differential $K$-groups
$\check{K}^0_{\mathcal L}(M)$ satisfy an axiomatic characterization,
as outlined in \cite[Section 7]{Bunke-Schick (2012)}. \\
3. Construct a topological pushforward in $\check{K}^0$. We recall that
the topological pushforward in $KK$-theory involves a Kasparov product
with a $\overline{\partial}$-operator \cite{Kasparov (1988)}. 

The paper has an appendix in which we prove a formula for the
Chern character of a superconnection in
relative cohomology. There is an application to eta forms.

More detailed descriptions of the content of the paper appear at the beginnings
of the sections.

We thank Dan Freed for discussions about
superconnections and differential $K$-theory, and for comments on an
earlier version of this paper.
We thank Jean-Pierre Bourguignon and Alan Weinstein for telling us
about diffeology. We also thank the referees for helpful comments.

\section{Pseudodifferential calculus} \label{sect1}

This section is devoted to functional analytic preliminaries.
Given a $\Z_2$-graded Hilbert space and an odd self-adjoint operator
$D$ that is $\theta$-summable for all $\theta > 0$, in Subsection
\ref{subsect1.1} we define order-$s$ Sobolev spaces $H^s$ and
order-$k$ pseudodifferential operators $op^k$. We prove basic
composition properties of the pseudodifferential operators. In Subsection
\ref{subsect1.2} we consider a space ${\mathcal P}$ of ``Dirac-type''
order-$1$ operators.  We prove in particular that ${\mathcal P}$ is
preserved by the addition of order-$0$ operators.  Subsection
\ref{subsect1.3} shows that for any $\epsilon >0$,
the map that sends $P$ to $e^{- \epsilon P^2}$ is smooth,
in the diffeological sense, as a map from ${\mathcal P}$ to the space
of trace-class operators. The heart of the proof is to justify
a Duhamel formula in this setting.

\subsection{Operator norms} \label{subsect1.1}

Let $H = H^+ \oplus H^-$ be a $\mathbb{Z}_2$-graded Hilbert space (possibly finite dimensional) with inner product $\langle \cdot, \cdot \rangle_H$.
Let ${\mathcal L}(H)$ denote the bounded operators on $H$,
with operator norm $\| \cdot \|$.
Let ${\mathcal L}^1(H)$ be the trace ideal of $H$, with its
norm $\| \cdot \|_{{\mathcal L}^1}$.
Let $D$ be an odd (with respect to the $\Z_2$-grading)
self-adjoint operator on $H$, possibly unbounded,
which is $\theta$-summable for all $\theta > 0$, i.e.
\begin{equation} \label{1.1}
\Tr e^{-\theta D^2} < \infty.
\end{equation}
In particular, $D^2$ has a discrete spectrum.
Let $P_{\Ker(D^2)}$ be orthogonal projection onto $\Ker(D^2)$.
Define
\begin{equation} \label{1.2}
|D| = \sqrt{D^2} + P_{\Ker(D^2)}.
\end{equation}
If $D$ is invertible then $|D|$ has the usual meaning, but for us $|D|$
is always a strictly positive operator.

For $s \in \Z$ nonnegative, put $H^s = \Dom \left( |D|^s \right)$,
with the inner product
\begin{equation} \label{1.3}
\langle v_1, v_2 \rangle_{H^s} =
\langle |D|^s v_1, |D|^s v_2 \rangle_{H}.
\end{equation}
For $s \in \Z$ negative, put $H^s = \left( H^{-s} \right)^*$.
Put $H^\infty = \bigcap_{s \ge 0} H^s$, a dense subspace of $H$.

Following
\cite[Appendix B]{Connes-Moscovici (1995)},
let $op^k$ be the closed operators $F$ such that
\begin{enumerate}
\item  $H^\infty \subset \Dom(F)$,
\item $F( H^\infty) \subset H^\infty$, and
\item
For all $s \in \Z$,  the operator $F : H^\infty \rightarrow H^\infty$ extends
to a bounded operator from $H^s$ to $H^{s-k}$.
\end{enumerate}

Let $|F|_{k,s}$ be the operator norm for $F : H^s \rightarrow H^{s-k}$.
Then $op^k$ is a Fr\'echet space with respect to the norms
$|F|_{k,s}$. We take a product of operators to act from right to left.
Using the isometric isomorphism $|D|^{-s} :
H^0 \rightarrow H^s$, if $F \in op^k$ then
\begin{equation} \label{1.4}
|F|_{k,s} = \parallel |D|^{s-k} F
|D|^{-s} \parallel.
\end{equation}

Let ${\mathcal L}^{fr}_\infty$ be the ideal of
finite rank operators, i.e. the set of operators $T$ on $H$
that can be expressed as
\begin{equation} \label{1.4.5}
T(v) = \sum_i \xi_i \langle \eta_i, v \rangle_H,
\end{equation}
where the sum is finite and $\xi_i, \eta_i \in H^\infty$.
Then ${\mathcal L}^{fr}_\infty \subset \bigcap_{k \in \Z} op^k$.

\begin{lemma} \label{lem1}
\begin{enumerate}
\item If $F_1 \in op^{k_1}$ and $F_2 \in op^{k_2}$
then $F_1 F_2 \in op^{k_1+k_2}$ and
\begin{equation} \label{1.5}
|F_1 F_2 |_{k_1+k_2, s} \le |F_1|_{k_1, s-k_2} |F_2|_{k_2, s}.
\end{equation}
\item If $F \in op^k$, and $F : H^s \rightarrow H^{s-k}$ is an isomorphism
for each $s \in \Z$, then $F^{-1} \in op^{-k}$. \\
\item
If $F \in op^0$ then its adjoint $F^*$ in $B(H)$ satisfies $F^* \in op^0$.
\end{enumerate}
\end{lemma}
\begin{proof}
(1). The proof is straightforward. \\
(2). By the bounded inverse theorem, $F^{-1} : H^{s-k} \rightarrow H^s$
is bounded for each $s \in \Z$. In particular, $H^\infty \subset \Dom(F^{-1})$
and $F^{-1}(H^\infty) \subset H^\infty$. \\
(3). If $v \in H^s$ then for all $w \in H^\infty$, we have
\begin{equation} \label{1.6}
\langle F w, v \rangle_H =
\langle
\left( |D|^{-s} F |D|^{s} \right)
|D|^{-s} w,
|D|^{s} v \rangle_H,
\end{equation}
showing that
\begin{equation} \label{1.7}
F^* v =
|D|^{-s}
\left( |D|^{-s} F |D|^{s} \right)^*
|D|^{s} v.
\end{equation}
In particular,
\begin{equation} \label{1.8}
|D|^{s}
F^* v =
\left( |D|^{-s} F |D|^{s} \right)^*
|D|^{s} v \in H,
\end{equation}
showing that $F^*(H^s) \subset H^s$, with
\begin{equation} \label{1.9}
F^* = |D|^{-s}
\left( |D|^{-s} F |D|^{s} \right)^*
|D|^{s}
\end{equation}
being a bounded operator on $H^s$.
\end{proof}

\begin{example} \label{ex1}
Suppose that $H$ is finite dimensional.  Then
$H^\infty = H$ and $H^s = H$ for all $s \in \Z$.
Also, $op^k = B(H)$ for all $k \in \Z$.
\end{example}

\begin{example} \label{ex2}
Let $Z$ be a compact Riemannian manifold.  Let $V$ be a Clifford
module over $Z$, equipped with a compatible connection.
Let ${\mathcal D}^{\frac12}$ be the half-density line bundle
on $Z$. Put $H = L^2(Z; {\mathcal D}^{\frac12} \otimes V)$.
Let $D$ be the Dirac-type operator on $H$. Then
$H^{\infty} = C^\infty(Z; {\mathcal D}^{\frac12} \otimes V)$ and
$H^s = H^s(Z; {\mathcal D}^{\frac12} \otimes V)$.
A pseudodifferential operator of order $k$ gives an element of
$op^k$.
\end{example}

\subsection{The space of Dirac-type operators} \label{subsect1.2}

\begin{definition} \label{def1}
$\mathcal{P}$ is the space of odd self-adjoint operators $P \in op^1$ such that
$|P|^{-1} \in op^{-1}$.
\end{definition}

In Section \ref{sect2}, the elements of ${\mathcal P}$ will become
the possible degree-$0$ terms of the superconnection.

Recall that $|P|$ is defined to be $1$ on $\Ker(P^2)$.
If $P \in {\mathcal P}$ then $P^2 |P|^{-1} \in op^1$. As $|P| -
P^2 |P|^{-1} \in {\mathcal L}^{fr}_\infty$, it follows that
$|P| \in op^1$.
Then for any $s \in \Z$, the operators $|D|^{s}
|P|^{-s}$ and $|P|^s
|D|^{-s}$ lie in $op^0$ and, in particular, are bounded
on $H$. It follows that if we defined $H^s$ using $P$ instead of $D$
then we would get the same spaces. If we think of $D$ as being a given
Dirac operator then we can think of ${\mathcal P}$ as being a collection of
Dirac-type operators.  For example, if $D$ is the operator of
Example \ref{ex2}, and $P$ is the operator arising from a different
Riemannian metric on $Z$ and a different Clifford connection on $V$,
then $P \in {\mathcal P}$.

\begin{lemma} \label{lem2}
\begin{enumerate}
\item Any $P \in \mathcal{P}$ is $\theta$-summable for all $\theta > 0$.
\item If $P \in \mathcal{P}$, and $Q \in op^0$ is odd and self-adjoint, then
$P+Q$ is $\theta$-summable for all $\theta > 0$. More precisely, for
any $\epsilon \in (0,1)$, we have
\begin{equation} \label{1.10}
\Tr e^{- \theta (P+Q)^2} \le
e^{\theta \left( \epsilon^{-2} - 1 \right) \| Q \|^2} \cdot
\Tr e^{- \theta (1 - \epsilon^2) P^2}.
\end{equation}
\item Given $P \in \mathcal{P}$, $F_1 \in op^{k_1}$, $F_2 \in op^{k_2}$ and
$\epsilon >0$,
for every $t \ge \epsilon$ we have
\begin{equation} \label{1.11}
\| F_1 e^{-tP^2} F_2 \|_{\mathcal{L}^1} \le C(\epsilon, P, k_1, k_2)
|F_1|_{k_1,k_1} |F_2|_{k_2, 0}.
\end{equation}
and
\begin{equation} \label{trace}
\Tr \left( F_1 e^{-tP^2} F_2 \right) =
\Tr \left( e^{-tP^2} F_2 F_1 \right) =
\Tr \left( F_2 F_1 e^{-tP^2} \right).
\end{equation}
\end{enumerate}
\end{lemma}
\begin{proof} (1). Since $|P|^{-1} D^2
|P|^{-1}$ lies in
$op^0$, it is bounded on $H$. Hence there is some $C < \infty$ so that
$D^2 \le C (I + P^2)$. Thus $P^2 \ge C^{-1} D^2 - I$, so
$P$ is $\theta$-summable. \\
(2). We follow the method of proof of \cite[Theorem C]{Getzler-Szenes (1989)}.
For any $\epsilon > 0$, we have
\begin{equation} \label{1.12}
0 \le (\epsilon P + \epsilon^{-1} Q)^2 = \epsilon^2 P^2 + (PQ+QP) +
\epsilon^{-2} Q^2.
\end{equation}
If $\epsilon \in (0,1)$ then
\begin{align} \label{1.13}
- \theta (P+Q)^2 = & - \theta (P^2 + PQ+QP+Q^2) \\
\le &
- \theta (1 - \epsilon^2) P^2 + \theta (\epsilon^{-2} - 1) Q^2 \notag \\
\le &
- \theta (1 - \epsilon^2) P^2 + \theta (\epsilon^{-2} - 1) \| Q \|^2. \notag
\end{align}
The claim follows.
\\
(3). We have
\begin{align} \label{1.14}
F_1 e^{-tP^2} F_2 = &
\left( F_1 |D|^{-k_1} \right) \cdot
\left( |D|^{k_1} |P|^{-k_1}  \right) \cdot
e^{- \: \frac{\epsilon}{2} P^2} \cdot \\
& \left( |P|^{k_1} e^{- (t - \frac{\epsilon}{2})P^2}
|P|^{k_2} \right) \cdot
\left( |P|^{-k_2} |D|^{k_2}
\right) \cdot \notag \\
& \left( |D|^{-k_2} F_2 \right). \notag
\end{align}
By assumption, each of the six factors in (\ref{1.14}) is bounded on $H$.
Part (1) shows that
$e^{- \: \frac{\epsilon}{2} P^2}$ is trace class.  Hence the product
is trace class. Using (\ref{1.4}), we obtain
\begin{align} \label{1.15}
\| F_1 e^{-tP^2} F_2 \|_{{\mathcal L}^1} \le &
| F_1 |_{k_1,k_1} \cdot
\| |D|^{k_1} |P|^{-k_1} \| \cdot \\
&
\| e^{- \: \frac{\epsilon}{2} P^2} \|_{{\mathcal L}^1} \cdot
\| |P|^{k_1} e^{- (t - \frac{\epsilon}{2})P^2}
|P|^{k_2} \| \cdot
\notag \\
& \| |P|^{-k_2} |D|^{k_2} \| \cdot
| F_2 |_{k_2,0}. \notag
\end{align}
From the spectral theorem,
\begin{equation} \label{1.16}
\| |P|^{k_1} e^{- (t - \frac{\epsilon}{2})P^2}
|P|^{k_2} \| \le
\sup_{r \in \R} \left( (1+r^2)^{\frac{k_1 + k_2}{2}}
e^{- \frac{\epsilon}{2} r^2} \right) < \infty.
\end{equation}

Next, we can write
\begin{align}
\Tr \left( F_1 e^{-tP^2} F_2 \right) = &
\Tr \left( F_1 e^{- \frac{\epsilon}{4} P^2} \cdot e^{- \left(t -
\frac{\epsilon}{2} \right) P^2} \cdot
e^{- \frac{\epsilon}{4} P^2} F_2 \right) \\
= & \Tr \left(  e^{- \left(t -
\frac{\epsilon}{2} \right) P^2} \cdot
e^{- \frac{\epsilon}{4} P^2} F_2 \cdot F_1 e^{- \frac{\epsilon}{4} P^2}
\right) \notag \\
= & \Tr \left( e^{- \frac{\epsilon}{4} P^2} e^{- \left(t -
\frac{\epsilon}{2} \right) P^2}
e^{- \frac{\epsilon}{4} P^2} F_2 F_1
\right) \notag \\
= & \Tr \left( e^{- t P^2} F_2 F_1 \right). \notag
\end{align}
Similarly,
\begin{align}
\Tr \left( F_1 e^{-tP^2} F_2 \right) = &
\Tr \left( F_1 e^{- \frac{\epsilon}{4} P^2} \cdot e^{- \left(t -
\frac{\epsilon}{2} \right) P^2} \cdot
e^{- \frac{\epsilon}{4} P^2} F_2 \right) \\
= &
\Tr \left(
e^{- \frac{\epsilon}{4} P^2} F_2 \cdot
F_1 e^{- \frac{\epsilon}{4} P^2} \cdot e^{- \left(t -
\frac{\epsilon}{2} \right) P^2} \right) \notag \\
= &
\Tr \left(
F_2 F_1 e^{- \frac{\epsilon}{4} P^2} e^{- \left(t -
\frac{\epsilon}{2} \right) P^2} e^{- \frac{\epsilon}{4} P^2} \right)
\notag \\
= &
\Tr \left(
F_2 F_1 e^{- t P^2} \right). \notag
\end{align}
This proves the claim.
\end{proof}

\begin{proposition} \label{prop1}
An odd self-adjoint operator $P \in op^1$  is in
$\mathcal{P}$ if and only if there is some odd self-adjoint operator
$Q  \in op^{-1}$ with
$PQ-I \in op^{-1}$ and $QP-I \in op^{-1}$.
\end{proposition}

To prove Proposition \ref{prop1}, we begin with some lemmas.

\begin{lemma} \label{lem3}
Suppose that  $P \in op^1$ is an odd self-adjoint Fredholm operator.
Then the spectrum of $P$ is disjoint
from $[-\epsilon, 0) \cup(0, \epsilon]$, for some $\epsilon > 0$.
Put $R = f(P)$, where
\begin{equation} \label{1.17}
f(t) = \begin{cases}
         1/t & \mbox{if } |t|>\epsilon, \\
         1 & \mbox{otherwise}.
       \end{cases}
\end{equation}
(This is independent of the choice of such $\epsilon$.)
Then $|P|^{-1} \in op^{-1}$ if and only if $ R\in op^{-1}$.
\end{lemma}
\begin{proof}
Suppose that $|P|^{-1} \in op^{-1}$. Then $P |P|^{-2} \in op^{-1}$.
As $R - P |P|^{-2} \in {\mathcal L}^{fr}_\infty$, it follows that
$R \in op^{-1}$.

Now suppose that $R\in op^{-1}$. Since $R^{-1} - P
\in {\mathcal L}^{fr}_\infty$,
it follows that $R^{-1} \in op^1$. Given $s \in \mathbb{Z}$,
write
\begin{equation} \label{1.18}
|D|^{s+1} |P|^{-1} |D|^{-s} = \left( |D|^{s+1} R^{s+1} \right) \cdot
\left( R^{-s-1} |P|^{-1} R^{s} \right) \cdot
\left( R^{-s} |D|^{-s} \right).
\end{equation}
As $|D|^{s+1} R^{s+1} \in op^0$ and
$R^{-s} |D|^{-s} \in op^0$, they are bounded operators.
Since $R = f(P)$ we have $R^{-s-1} |P|^{-1} R^{s} = R^{-1} |P|^{-1}$,
which is bounded. Hence $|P|^{-1} \in op^{-1}$.
\end{proof}

\begin{lemma} \label{lem4}
If $Q_0$  is an odd self-adjoint element of $op^{-1}$
such that $Q_0P-1 \in {\mathcal L}^{fr}_\infty(H)$,
then $R-Q_0 \in {\mathcal L}^{fr}_\infty(H)$.
\end{lemma}
\begin{proof}
First,
$Q_0PR-R = (Q_0 P - I)R \in {\mathcal L}^{fr}_\infty(H)$.
Note that $I-PR$ is the projection on the kernel of $P$, hence in
${\mathcal L}^{fr}_\infty$.
Then $Q_0 -Q_0PR = Q_0 (I-PR) \in {\mathcal L}^{fr}_\infty$.
The lemma follows.
\end{proof}

\begin{lemma} \label{lem5}
If $A \in op^k$ is an even operator, with $k<0$, then
$A$ is a limit of even elements of ${\mathcal L}^{fr}_\infty$,
in the $op^0$-topology.
\end{lemma}
\begin{proof}
Write $A=B |D|^k$, with $B \in op^0$.
Put
\begin{equation} \label{1.19}
   \chi_n(t)=\begin{cases}
               1 & \mbox{if } t\le n, \\
               0 & \mbox{otherwise}.
             \end{cases}
\end{equation}
Put $A_n = A \chi_n(|D|)$.
Then for every $s \in \Z$,
\begin{align} \label{1.20}
|A-A_n|_{0,s} = &
\| |D|^s B |D|^k
\cdot (1-\chi_n)(|D|) \cdot |D|^{-s} \| \\
= & \| |D|^s B  |D|^{-s}
\cdot |D|^k (1-\chi_n)(|D|) \| \notag \\
\le & \| |D|^s B  |D|^{-s} \| \cdot \| |D|^k (1-\chi_n)(|D|) \| \notag \\
= & |B|_{0,s} \cdot \||D|^k (1-\chi_n)(|D|) \| \le |B|_{0,s} n^k. \notag
\end{align}
Hence $\lim_{n \rightarrow \infty} |A-A_n|_{0,s} = 0$.
This proves the lemma.
\end{proof}

\begin{proof}[Proof of Proposition \ref{prop1}]
Suppose that $P \in {\mathcal P}$. Then we can take $Q$ to be the
operator $R$ of Lemma \ref{lem3}.

Conversely, suppose that we have an odd self-adjoint operator
$P \in op^1$,  and $Q  \in op^{-1}$ is an odd self-adjoint operator such that
$PQ-I \in op^{-1}$ and $QP-I \in op^{-1}$.
It follows that $P$ is Fredholm, and we therefore can define
$R$ as in (\ref{1.17}).
By the proof of Lemma \ref{lem5}, for every $N \in \Z^+$, we can write
\begin{equation} \label{1.21}
QP-I= F_N - A_N,
\end{equation}
where $F_N$ and $A_N$ are even operators, $F_N \in {\mathcal L}^{fr}_\infty(H)$
and $C_N = \max \{|A_N|_{0, s} \: : \: |s| \le N\}<1$.
Then $I-A_N$ is an even invertible operator on $H$ and for all $s \in \Z$ with
$|s| \le N$,
\begin{equation} \label{1.22}
|(I-A_N)^{-1}|_{0, s} = \left| \sum_{m=0}^{\infty} A_N^m \right|_{0, s} \le
\frac{1}{1-C_N}.
\end{equation}
Equation (\ref{1.21}) implies that
\begin{equation} \label{1.23}
(I-A_N) R - Q = Q(PR-I) - F_N R.
\end{equation}
Multiplying on the left by $(I-A_N)^{-1}$ shows that
\begin{equation} \label{1.23.5}
R- (1-A_N)^{-1}Q = (I - A_N)^{-1} ( Q (PR - I) - F_NR).
\end{equation}
Note that $Q (PR - I) - F_NR \in {\mathcal L}^{fr}_\infty$.
Given $s \in \Z$, if $N$ is sufficiently large then the
right-hand side of (\ref{1.23.5}) is a bounded operator from
$H^s$ to $H^{s+1}$.
Hence the operator
\begin{equation} \label{1.24}
|D|^{s+1} R
|D|^{-s} - \left( |D|^{s+1} (1-A_N)^{-1}
|D|^{-s-1} \right) \left( |D|^{s+1} Q
|D|^{-s} \right),
\end{equation}
defined originally on $H^\infty$,
extends to an odd bounded operator on $H$.
For $N$ sufficiently large, we know that $|D|^{s+1} (I - A_N)^{-1}
|D|^{-s-1}$ extends to a bounded operator on $H$. Also,
for any $s \in \Z$, the operator $|D|^{s+1} Q
|D|^{-s}$ extends to a bounded operator on $H$.
Therefore $|D|^{s+1} R
|D|^{-s}$ extends to an odd bounded operator on $H$
and $R \in op^{-1}$.
Lemma \ref{lem3} now implies the proposition.
\end{proof}

\begin{corollary} \label{cor1} If $P \in \mathcal{P}$, and $A \in op^0$ is an
odd self-adjoint operator,
then $P+A \in \mathcal{P}$.
\end{corollary}
\begin{proof}
From Proposition \ref{prop1},
there is some odd self-adjoint $Q \in op^{-1}$ so that
$PQ - I \in op^{-1}$ and $QP - I \in op^{-1}$.
Then $(P+A)Q- I \in op^{-1}$ and $Q(P+A) - I \in op^{-1}$. The
corollary now follows from Proposition \ref{prop1}.
\end{proof}

\subsection{Duhamel formula} \label{subsect1.3}

We say that a map from $\R^n$ to ${\mathcal P}$ is smooth if
the composite map $\R^n \rightarrow {\mathcal P} \subset op^1$ is
smooth with respect to the Fr\'echet topology on $op^1$.

\begin{proposition} \label{prop2}
If $f \colon \mathbb{R}^n \to \mathcal{P}$ is smooth, then $x \mapsto e^{-\epsilon f(x)^2}$ is a smooth
map from $\mathbb{R}^n$ to $\mathcal{L}^1(H)$.
\end{proposition}
\begin{proof}
Consider first the case when $n=1$.
Let $f : \R \rightarrow \mathcal{P}$ be a smooth
map, parametrized by $u \in \R$. We claim that for $u_1 < u_2$, we have
\begin{equation} \label{1.25}
e^{- \epsilon f(u_2)^2} - e^{- \epsilon f(u_1)^2} = -
\int_{u_1}^{u_2} \int_0^{\epsilon}
e^{- \sigma f(v)^2} \frac{df(v)^2}{dv} e^{- (\epsilon - \sigma)
f(v)^2} \: d\sigma \: dv.
\end{equation}
To give meaning to the integral over $\sigma$, we rewrite it as
\begin{align} \label{1.26}
& \int_0^{\epsilon}
e^{- \sigma f(v)^2} \frac{df(v)^2}{dv} e^{- (\epsilon - \sigma)
f(v)^2} \: d\sigma = \\
& \int_0^{\frac{\epsilon}{2}}
e^{- \sigma f(v)^2} \left( \frac{df(v)^2}{dv} e^{- \frac{\epsilon}{2}
f(v)^2} \right)
e^{- \left( \frac{\epsilon}{2} - \sigma \right)
f(v)^2} \: d\sigma + \notag \\
& \int_{\frac{\epsilon}{2}}^{\epsilon}
e^{- \left( \sigma - \frac{\epsilon}{2} \right) f(v)^2}
\left( e^{- \frac{\epsilon}{2}
f(v)^2} \frac{df(v)^2}{dv} \right) e^{- (\epsilon - \sigma)
f(v)^2} \: d\sigma. \notag
\end{align}
Using Lemma \ref{lem2}.(3),
the integrands in the last two integrals are
continuous as maps from $\left[ 0, \frac{\epsilon}{2} \right]$
(or $\left[ \frac{\epsilon}{2}, \epsilon \right]$) to
${\mathcal L}^1(H)$.

To prove (\ref{1.25}), we first prove the corresponding statement for
resolvents.  For $\lambda \in \C - \R^{\ge 0}$,  and $- \infty < v_1 < v_2
< \infty$,
we have
\begin{align} \label{1.27}
& (\lambda - f(v_2)^2)^{-1} - (\lambda - f(v_1)^2)^{-1} = \\
& (\lambda - f(v_2)^2)^{-1}  \cdot
(f(v_2)^2 - f(v_1)^2) \cdot
(\lambda - f(v_1)^2)^{-1}. \notag
\end{align}
Then\begin{align} \label{1.28}
& \frac{(\lambda - f(v_2)^2)^{-1} - (\lambda - f(v_1)^2)^{-1}}{v_2 - v_1} \: -  \\
& (\lambda - f(v_2)^2)^{-1}  \cdot
\frac{df(v)^2}{dv} \Big|_{v=v_1} \cdot
(\lambda - f(v_1)^2)^{-1} = \notag \\
&  (\lambda - f(v_2)^2)^{-1} \cdot
\left( \frac{f(v_2)^2 - f(v_1)^2}{v_2 - v_1} -
\frac{df(v)^2}{dv} \Big|_{v=v_1}
\right) \cdot \notag \\
& (\lambda - f(v_1)^2)^{-1}. \notag
\end{align}
By assumption, $f \: : \: \R \rightarrow op^1$ is differentiable, so
$f^2 \: : \: \R \rightarrow op^2$ is differentiable.
Using the fact that $H^s$ can be defined using $f(v) \in
{\mathcal P}$ instead of $D$, it follows that
$(\lambda - f(v)^2)^{-1} \in op^{-2}$.
From (\ref{1.27}),
\begin{equation} \label{1.29}
(\lambda - f(\cdot)^2)^{-1} \: : \:
\R \rightarrow op^{-2}
\end{equation}
is continuous.
Then (\ref{1.28}) implies that (\ref{1.29})
is differentiable and the derivative is given by
\begin{equation} \label{1.30}
\frac{d}{dv} (\lambda - f(v)^2)^{-1} =
(\lambda - f(v)^2)^{-1} \frac{df(v)^2}{dv} (\lambda - f(v)^2)^{-1}.
\end{equation}
Hence
\begin{align} \label{1.31}
& (\lambda - f(u_2)^2)^{-1} - (\lambda - f(u_1)^2)^{-1} = \\
& \int_{u_1}^{u_2}
(\lambda - f(v)^2)^{-1} \frac{df(v)^2}{dv} (\lambda - f(v)^2)^{-1} \: dv,
\notag
\end{align}
where the integrand is a continuous map from $[u_1, u_2]$ to
$op^{-2}$.

Put $\Gamma = \{(|t| - 1, t) \: : \: t \in \R \}$, a parametrized curve in
the complex plane. By the spectral theorem,
\begin{equation} \label{1.32}
e^{- \epsilon f(u)^2} = \frac{1}{2 \pi i} \int_{\Gamma}
e^{- \epsilon \lambda} (\lambda - f(u)^2)^{-1} \: d\lambda.
\end{equation}
Then
\begin{align} \label{1.33}
& e^{- \epsilon f(u_2)^2} - e^{- \epsilon f(u_1)^2} = \\
& \frac{1}{2 \pi i} \int_{\Gamma} \int_{u_1}^{u_2}
e^{- \epsilon \lambda}
(\lambda - f(v)^2)^{-1} \frac{df(v)^2}{dv} (\lambda - f(v)^2)^{-1}
\: dv \: d\lambda
 = \notag \\
& \frac{1}{2 \pi i} \int_{\Gamma} \int_{u_1}^{u_2}
e^{- \epsilon \lambda}
\frac{I + f(v)^2}{\lambda - f(v)^2}
(I + f(v)^2)^{-1}
 \frac{df(v)^2}{dv}
(I + f(v)^2)^{-1} \cdot \notag \\
& \frac{I + f(v)^2}{\lambda - f(v)^2}.
\: dv \: d\lambda \notag
\end{align}
The spectral theorem gives a uniform bound on
$\Big\| \frac{I + f(v)^2}{\lambda - f(v)^2} \Big\|$ for
$v \in \R$ and $\lambda \in \Gamma$.
Also, $\| I + f(v)^2)^{-1}
 \frac{df(v)^2}{dv}
(I + f(v)^2)^{-1} \|$ is uniformly bounded for $v \in [u_1, u_2]$.
Combined with the exponential decay
of $e^{- \epsilon \lambda} = e^{- \epsilon (|t| - 1 + it)}$
as $t \rightarrow \pm \infty$,
one can justify switching the order of integration to obtain
\begin{align} \label{1.34}
& e^{- \epsilon f(u_2)^2} - e^{- \epsilon f(u_1)^2} = \\
& \frac{1}{2 \pi i}  \int_{u_1}^{u_2} \int_{\Gamma}
e^{- \epsilon \lambda}
(\lambda - f(v)^2)^{-1} \frac{df(v)^2}{dv} (\lambda - f(v)^2)^{-1}
\: d\lambda \: dv. \notag
\end{align}

We claim that
\begin{align} \label{1.35}
& \frac{1}{2 \pi i} \int_{\Gamma}
e^{- \epsilon \lambda}
(\lambda - f(v)^2)^{-1} \frac{df(v)^2}{dv} (\lambda - f(v)^2)^{-1}
\: d\lambda = \\
& - \int_0^{\epsilon} e^{- \sigma f(v)^2} \frac{df(v)^2}{dv}
e^{- (\epsilon - \sigma) f(v)^2} \: d\sigma. \notag
\end{align}
To see this, let $e_1$ and $e_2$ be eigenfunctions of $f(v)^2$, with
eigenvalues $\lambda_1$ and $\lambda_2$, respectively. Then
 \begin{align} \label{1.36}
& \frac{1}{2 \pi i} \int_{\Gamma}
e^{- \epsilon \lambda}  \langle e_1,
(\lambda - f(v)^2)^{-1} \frac{df(v)^2}{dv} (\lambda - f(v)^2)^{-1}
 \: e_2 \rangle
\: d\lambda \: = \\
& \frac{1}{2 \pi i} \int_{\Gamma}
e^{- \epsilon \lambda} (\lambda - \lambda_1)^{-1}
(\lambda - \lambda_2)^{-1} \: \langle e_1,
\frac{df(v)^2}{dv}  \: e_2 \rangle \: d\lambda
\:  =  \notag \\
& \frac{e^{-\epsilon \lambda_2} - e^{-\epsilon \lambda_1}}{\lambda_2 -
\lambda_1} \langle e_1,
\frac{df(v)^2}{dv}  \: e_2 \rangle, \notag
\end{align} 
where
$\frac{e^{-\epsilon \lambda_2} - e^{-\epsilon \lambda_1}}{\lambda_2 -
\lambda_1}$ is taken to be $- \epsilon e^{- \epsilon \lambda_1}$ if
$\lambda_2 = \lambda_1$.
On the other hand,
\begin{align} \label{1.37}
& \int_0^{\epsilon} \langle e_1, e^{- \sigma f(v)^2} \frac{df(v)^2}{dv}
e^{- (\epsilon - \sigma) f(v)^2} e_2 \rangle \: d\sigma = \\
& \int_0^{\epsilon} e^{- \sigma \lambda_1} e^{- (\epsilon - \sigma) \lambda_2}
\: \langle e_1,  \frac{df(v)^2}{dv}
 e_2 \rangle \: d\sigma  \: = \notag \\
& - \: \frac{e^{-\epsilon \lambda_2} - e^{-\epsilon \lambda_1}}{\lambda_2 -
\lambda_1} \langle e_1,  \frac{df(v)^2}{dv}
 e_2 \rangle, \notag
\end{align}
which proves the claim.

This proves (\ref{1.25}). It follows that
$u \rightarrow e^{- \epsilon f(u)^2}$
is differentiable as a map into ${\mathcal L}^1(H)$, with derivative
\begin{equation} \label{1.38}
\frac{d}{du} e^{- \epsilon f(u)^2}
= - \int_0^{\epsilon}
e^{- \sigma f(u)^2} \frac{df(u)^2}{du} e^{- (\epsilon - \sigma)
f(u)^2} \: d\sigma.
\end{equation}
If $f \: : \: \R^n \rightarrow {\mathcal P}$ is
a smooth map then precomposing with smooth maps $\R \rightarrow \R^n$,
we see that $x \rightarrow e^{- \epsilon f(x)^2}$ is differentiable as
a map from $\R^n$ to ${\mathcal L}^1(H)$. By a similar argument, we can take
more derivatives and see that $x \rightarrow e^{- \epsilon f(x)^2}$ is
a smooth map from $\R^n$ to ${\mathcal L}^1(H)$.
\end{proof}

\section{Superconnections} \label{sect2}

In this section we develop the theory of superconnections on Hilbert
bundles.
In Subsection \ref{subsect2.1} we define the relevant
class of Hilbert bundles and specify, in particular, the structure
group.  In Subsection \ref{subsect2.2} we define superconnections on
such Hilbert bundles,
using the pseudodifferential operators of the previous section.
We construct Chern characters and eta forms.  We prove additivity
results for eta forms.

\subsection{Structure group} \label{subsect2.1}

Put $G = U(H) \cap op^0$.

\begin{lemma} \label{lem6}
$G$ is a group.
\end{lemma}
\begin{proof}
If $g \in G$ then
Lemma \ref{lem1}.(3) gives that $g^* \in op^0$,
so $g^*$ is an inverse of $g$ in $G$.
\end{proof}

We now put a smooth structure on $G$. Since we will be considering
principal $G$-bundles over finite dimensional manifolds, it suffices
to give a notion of smooth maps from  open subsets of Euclidean spaces
to $G$, i.e. plots in the sense of diffeology.  A reference for
diffeology is the book \cite{Iglesias (2013)}. A brief introduction
is in \cite[Appendix A]{Blohmann-Weinstein (2013)}.

The smooth structure on $G$ that we take
is such that the adjoint action of $G$ on
$op^*$ is smooth and the action of $G$ on $H^*$ is smooth. The
precise definition is the following. (We define
smooth maps to $op^k$ using the Fr\'echet structure on $op^k$, and
to $H^s$ using the Hilbert space structure on $H^s$.)

In the rest of the paper, we fix
a number $K \in \N$; its role will eventually be to
bound the pseudodifferential order of the connection form.
Taking $K$ large allows more flexibility.  The results of the
paper, such as Theorems \ref{thm1} and
\ref{newtheorem2}, will hold independent
of $K$.

\begin{definition} \label{def2}
If $U$ is an open subset of
$\R^n$ then a map $g : U \rightarrow G$ is a plot if
\begin{enumerate}
\item For any smooth map $F : U \rightarrow op^k$, the maps
$gFg^{-1} : U \rightarrow op^k$ and $g^{-1}Fg : U \rightarrow op^k$ are
smooth.
\item For any smooth map $v : U \rightarrow H^s$, the maps
$g v : U \rightarrow H^s$ and $g^{-1} v : U \rightarrow H^s$ are smooth.
\item
There is a smooth map $X : U \rightarrow
(\R^n)^* \otimes op^K$
so that for any smooth map $v : U \rightarrow H^s$, we have
$g^{-1} d(gv) = dv + Xv$ in $\Omega^1 \left( U; H^{s-K} \right)$.
\end{enumerate}
\end{definition}

It is straightforward to see that this defines a diffeology on $G$.
Let $M$ be a manifold and let $q : P \rightarrow M$
be a smooth principal $G$-bundle
\cite[Chapter 8.11]{Iglesias (2013)}.
Form the associated Hilbert bundle $\mathcal{H} = P \times_G H$
\cite[Chapter 8.16]{Iglesias (2013)}.
We can find a covering $\{U_\alpha\}$ of $M$ by open sets, diffeomorphic
to open subsets of $\R^n$, so that
$q^{-1}(U_\alpha)$ is $G$-diffeomorphic to $U_\alpha \times G$
\cite[Chapter 8.13]{Iglesias (2013)}. Let $g_{\alpha \beta} :
U_\alpha \cap U_\beta \rightarrow G$ be the transition map.
By a connection on ${\mathcal H}$, we will mean a collection of
$1$-forms
$A_\alpha \in \Omega^1(U_\alpha; op^K)$
satisfying
\begin{equation} \label{2.1}
A_\alpha = g_{\alpha \beta}^{-1} A_\beta g_{\alpha \beta} +
X_{\alpha \beta},
\end{equation}
where $X_{\alpha \beta}$ comes from Definition \ref{def2}.(3).
Because of our notion of smooth structure on $G$,
we can talk about the space
$\bigoplus_{p \ge 0, k \in \Z} \Omega^p(M; op^k(\mathcal{H}))$
of smooth $op^*$-valued differential forms on $M$.
Since $G$ preserves the space ${\mathcal P}$ from Definition \ref{def1},
and acts smoothly on it,
we can also talk about $\Omega^0
\left( M; {\mathcal P}({\mathcal H}) \right)$.

\begin{example} \label{ex3}
If $H$ is finite dimensional then $G = U(N)$ and
${\mathcal H}$ is a finite dimensional unitary vector bundle.
\end{example}

\begin{example} \label{ex4}
Let $Z$ be a compact manifold.  Let ${\mathcal D}^{\frac12}$
be the half-density
line bundle on $Z$. Let $V$ be a Hermitian vector bundle
on $Z$. Let $L$ be the group of Hermitian isomorphisms of $V$ to
itself.  We do not assume that the elements of $L$ cover the identity
diffeomorphism of $Z$.
Put $H = L^2(Z; {\mathcal D}^{\frac12} \otimes V)$. Then
there is a homomorphism $L \rightarrow U(H)$.
We give $L$ the smooth topology.
Note that if $\dim(Z) > 0$
then the
homomorphism will not be continuous if we give $U(H)$ the
topology coming from the norm topology on $B(H)$.

Suppose now that $Z$ is even dimensional and
$V$ is a Clifford module.  In particular, $Z$
acquires a Riemannian metric.  Let $D$ be the associated Dirac-type
operator. Putting $G = U(H) \cap op^0$ as before, with its
diffeological structure,
there is a homomorphism $\rho : L \rightarrow G$
that is smooth in the sense that if $U$ is an open subset of $\R^n$, and
$\alpha : U \rightarrow L$ is a smooth map, then $\rho \circ \alpha$
is a plot for $G$.

Now let $\pi : Q \rightarrow M$ be a fiber bundle with connected base $M$ and
compact even dimensional fibers.
Let $E$ be a Hermitian vector bundle on
$Q$. Given $m \in M$, put $Z_m = \pi^{-1}(m)$ and put $V_m = E \Big|_{Z_m}$.
Choose $m_0 \in M$ and let $L$ be the Hermitian isomorphisms of
$V_{m_0}$.
Then $(Q,E)$ is associated to some
principal $L$-bundle $P \rightarrow M$, using
the action of $L$ on $V_{m_0}$.

Suppose that $E$ is a Clifford
module with connection,
in the sense of
\cite[Section 10.2]{Berline-Getzler-Vergne (1996)}.
Construct $H$ and $D$ as above, using $Z_{m_0}$ and $V_{m_0}$.
Put ${\mathcal H}_m = L^2(Z_m; {\mathcal D}_m^{\frac12} \otimes V_m)$.
Then $\{{\mathcal H}_m\}_{m \in M}$ are the fibers of a
Hilbert bundle ${\mathcal H}$ associated to $P$ using the
representation $\rho$. The smooth sections of ${\mathcal H}
\rightarrow M$ are the
same as the smooth sections of $({\mathcal D}^VQ)^{\frac12} \otimes E
\rightarrow Q$,
where $({\mathcal D}^VQ)^{\frac12}$ is the line bundle on $Q$ of
vertical half-densities.
\end{example}

\subsection{Chern character and eta form} \label{subsect2.2}

\begin{definition} \label{def3}
Let ${\mathcal H} = {\mathcal H}^+ \oplus {\mathcal H}^-$
be a $\Z_2$-graded Hilbert bundle on $M$, in the sense of
the previous subsection.
A superconnection on ${\mathcal H}$ is a sum
\begin{equation} \label{2.2}
\A =\A_{[0]}+\A_{[1]}+ \A_{[2]} +\ldots
\end{equation}
where
\begin{enumerate}
\item
$\A_{[0]} \in \Omega^0(M; \mathcal{P})$,
\item
$\A_{[1]}$ is a connection on $\mathcal{H}$
and
\item For $i \ge 2$, we have
$\A_{[i]} \in \Omega^i(M, op^{k_i}(\mathcal{H}))$ for some $k_i \in \Z$,
and $\A_{[i]}$ has total parity $-1$ (including the parity $(-1)^i$ of
the exterior algebra part).
\end{enumerate}
\end{definition}

\begin{example} \label{ex5}
In the fiber bundle setting of Example \ref{ex4},
the Bismut superconnection
$\A_{Bismut}$
\cite[Section 10.3]{Berline-Getzler-Vergne (1996)} is an example of a
superconnection on ${\mathcal H}$.
\end{example}

As usual, we define $e^{- \: \A^2}$ by doing a Duhamel expansion
around $e^{- \: \A_{[0]}^2}$.  Because of the nilpotency of
$\Omega^{\ge 1}(M)$, the expansion has a finite number of terms.

\begin{lemma} \label{lem7}
\begin{enumerate}
\item
For any $X \in \Omega^*(M; op^*({\mathcal H}))$, we have that
$X e^{- \: \A^2}$ and
$e^{- \: \A^2} X$ lie in $\Omega^*(M; {\mathcal L}^1({\mathcal H}))$, and
\begin{equation} \label{closed}
d \Str \left( X e^{- \: \A^2} \right) =
\Str \left( [\A, X] e^{- \: \A^2} \right).
\end{equation}
\item $\Ch(\A) = \Str e^{- \: \A^2}$ is a closed form on $M$.
\end{enumerate}
\end{lemma}
\begin{proof}
Expanding $e^{- \: \A^2} \in \Omega^*(M; op^*)$ around
$e^{- \: \A_{[0]}^2}$ in a Duhamel expansion shows that the component in
$\Omega^i(M; op^*)$ is a finite sum of terms of the form
\begin{equation} \label{2.3}
\int_{\Delta_k} e^{- \: t_0 \: \A_{[0]}^2} F_1 e^{- \: t_1 \: \A_{[0]}^2}
F_2 \ldots F_k e^{- \: t_k \: \A_{[0]}^2},
\end{equation}
where
\begin{equation} \label{2.4}
\Delta_k = \{ (t_0, \ldots, t_k) \in \R^{k+1} \: : \:
\sum_{j=0}^k t_j = 1 \}
\end{equation}
and each $F_j$ lies in $\Omega^{\ge 1}(M; op^*)$.
For any $(t_0, \ldots, t_k) \in \Delta_k$, we have
$t_j \ge \frac{1}{k+1}$ for some $j$. Thus
the integral in (\ref{2.3}) can be written as a finite sum of integrals
where in each integral, $t_j \ge \frac{1}{k+1}$ for some $j$.
The fact that
$X e^{- \: \A^2}$ and
$e^{- \: \A^2} X$ lie in $\Omega^*(M; {\mathcal L}^1({\mathcal H}))$
now follows from  Lemma \ref{lem2}.(3).
Using (\ref{trace}), equation (\ref{closed}) can be proved along the
same lines as the proof of
\cite[Lemma 9.15]{Berline-Getzler-Vergne (1996)}.
Finally, as in
\cite[Theorem 9.17(1)]{Berline-Getzler-Vergne (1996)},
part (2) of the lemma is an immediate consequence
of (1).
\end{proof}

\begin{lemma} \label{lem8}
Let $\{\A(t)\}_{t \in [0,1]}$ and $\{\widehat{\A}(t)\}_{t \in [0,1]}$
be two smooth $1$-parameter families of superconnections on ${\mathcal H}$
with $\A(0) = \widehat{\A}(0)$ and $\A(1) = \widehat{\A}(1)$.
Suppose that the two $1$-parameter families are homotopic
relative to the endpoints, in sense that there is
a smooth $2$-parameter family of superconnections
$\{\widetilde{\A}(s,t)\}_{s,t \in [0,1]}$ on ${\mathcal H}$
with $\widetilde{\A}(0,t) = \A(t)$, $\widetilde{\A}(1,t) =
\widehat{\A}(t)$, $\widetilde{\A}(s,0) = \A(0) = \widehat{\A}(0)$ and
$\widetilde{\A}(s,1) = \A(1) = \widehat{\A}(1)$.
Then
\begin{equation} \label{2.5}
\int_0^1 \Str \left( \frac{d \A(t)}{dt}
e^{- \: \A^2(t)} \right) dt =
\int_0^1 \Str \left( \frac{d \widehat{\A}(t)}{dt}
e^{- \: \widehat{\A}^2(t)} \right) dt
\end{equation}
in $\Omega^{odd}(M)/\Image(d)$.
\end{lemma}
\begin{proof}
Define a superconnection $\B$ on $[0,1] \times [0,1] \times M$ by
\begin{equation} \label{2.6}
\B = ds \wedge \partial_s + dt \wedge \partial_t + \widetilde{\A}(s,t).
\end{equation}
Then $\Ch(\B)$ is given by
\begin{align} \label{2.7}
\Ch(\B) = & \Ch(\widetilde{\A}(s,t)) - ds \wedge
\Str \left( \frac{d \widetilde{\A}(s,t)}{ds}
e^{- \: \widetilde{\A}^2(s,t)} \right) - \\
& dt \wedge
\Str \left( \frac{d \widetilde{\A}(s,t)}{dt}
e^{- \: \widetilde{\A}^2(s,t)} \right) + O(ds \wedge dt). \notag
\end{align}
From Lemma \ref{lem7}, $\Ch(\B)$ is closed on $[0,1] \times [0,1] \times M$.
Modulo $\Image(d_M)$, we have
\begin{align} \label{2.8}
\int_{\partial ([0,1] \times [0,1])} \Ch(\B) = &
\int_{[0,1] \times [0,1]} d_{[0,1] \times [0,1]} \Ch(\B) =
- \int_{[0,1] \times [0,1]} d_M \Ch(\B) \\
= & - d_M \int_{[0,1] \times [0,1]} \Ch(\B) = 0. \notag
\end{align}
The lemma follows.
\end{proof}

\begin{remark} \label{rem1}
We can weaken the homotopy hypothesis in Lemma \ref{lem8} to just assume that
there is a smooth $2$-parameter family
$\{{\mathcal O}(s,t)\}_{s,t \in [0,1]}$ in
$\Omega^0(M; {\mathcal P})$ with
${\mathcal O}(0,t) = \A(t)_{[0]}$, ${\mathcal O}(1,t) =
\widehat{\A}(t)_{[0]}$, ${\mathcal O}(s,0) = \A(0)_{[0]} =
\widehat{\A}(0)_{[0]}$ and
${\mathcal O}(s,1) = \A(1)_{[0]} = \widehat{\A}(1)_{[0]}$.
Then we can construct a $2$-parameter family
$\{\widetilde{\A}(s,t)\}_{0 \le s,t \le 1}$ as in the lemma.
\end{remark}

Let $\A_0$ and $\A_1$ be two superconnections on ${\mathcal H}$.
Let $\A_{0,[0]}$ and $\A_{1,[0]}$ denote their order-zero parts. Suppose that
$\A_{0,[0]}- \A_{1,[0]} \in \Omega^0(M; op^0({\mathcal H}))$.
For $t \in [0,1]$, put $\A(t) = (1-t)\A_0 + t\A_1$.
By Corollary \ref{cor1}, for any $t \in [0,1]$, we have
$\A(t)_{[0]} \in {\mathcal P}$.
Define $\eta(\A_0, \A_1) \in \Omega^{odd}(M)/ \Image(d)$ by
\begin{equation} \label{2.9}
\eta(\A_0, \A_1) = \int_0^1 \Str \left( \frac{d \A(t)}{dt}
e^{- \: \A^2(t)} \right) dt.
\end{equation}

\begin{lemma} \label{lem9}
\begin{equation} \label{2.10}
\Ch(\A_1)- \Ch(\A_0) = - d \eta(\A_0, \A_1).
\end{equation}
\end{lemma}
\begin{proof}
Consider the superconnection on $[0,1] \times M$ given by
\begin{equation} \label{2.11}
\B = dt \wedge \partial_t + \A(t).
\end{equation}
Then $\Ch(\B) \in \Omega^*([0,1] \times M)$ is given by
\begin{equation} \label{2.12}
\Ch(\B) = \Ch(\A(t)) - dt \wedge
\Str \left( \frac{d \A(t)}{dt}
e^{- \: \A^2(t)} \right).
\end{equation}
From Lemma \ref{lem7}, $\Ch(\B)$ is closed on $[0,1] \times M$.
This implies that
\begin{equation} \label{2.13}
\frac{d}{dt} \Ch(\A(t)) = - d \Str \left( \frac{d \A(t)}{dt}
e^{- \: \A^2(t)} \right)
\end{equation}
in $\Omega^*(M)$. The lemma
follows by integrating over $[0,1]$.
\end{proof}

\begin{lemma} \label{lem10}
Let $\A_0$, $\A_1$ and $\A_2$ be three superconnections on ${\mathcal H}$
such that $\A_{0,[0]} - \A_{1,[0]} \in \Omega^0(M; op^0({\mathcal H}))$ and
$\A_{1,[0]} - \A_{2,[0]} \in \Omega^0(M; op^0({\mathcal H}))$. Then
\begin{equation} \label{2.14}
\eta(\A_0, \A_1) +\eta(\A_1, \A_2) = \eta(\A_0, \A_2).
\end{equation}
\end{lemma}
\begin{proof}
For $s, t \in [0,1]$, put
\begin{align} \label{2.15}
& \widetilde{\A}(s,t) = \\
& \begin{cases}
(1-t-st) \A_0 + 2st \A_1 + (1-s) t \A_2
& \mbox{ if } 0 \le t \le \frac12, \\
(1-s)(1-t) \A_0 + 2s(1-t) \A_1 + (-s+t+ts) \A_2
& \mbox{ if } \frac12 \le t \le 1.
\end{cases} \notag
\end{align}
Then $\{\widetilde{\A}(0,t)\}_{t \in [0,1]}$
is the linear path from $\A_0$ to $\A_2$,
while $\{\widetilde{\A}(1,t)\}_{t \in [0,1]}$ is the concatenation
of the linear path from $\A_0$ to $\A_1$, with the linear path
from $\A_1$ to $\A_2$.
After reparametrizing $[0,1] \times [0,1]$ to make $\widetilde{\A}(s,t)$
smooth in $s$ and $t$, we can apply Lemma \ref{lem8}.
\end{proof}

Given a superconnection $\A$ on ${\mathcal H}$ and $t > 0$, define a new
superconnection by
\begin{equation} \label{2.16}
\A_t = t\A_{[0]} + \A_{[1]}+t^{-1} \A_{[2]}+\ldots.
\end{equation}

\begin{lemma} \label{lem11}
Suppose that there is some $c > 0$ so that $\A_{[0]}^2 \ge c^2 \Id$
fiberwise on ${\mathcal H}$.
Put
\begin{equation} \label{2.17}
\eta(\A, \infty) = \int_1^\infty \Str \left( \frac{d \A_t}{dt}
e^{- \: \A_t^2} \right) dt.
\end{equation}
Then
\begin{equation} \label{2.18}
\Ch(\A) = d \eta(\A, \infty)
\end{equation}
\end{lemma}
\begin{proof}
Given $m \in M$, let $\{ \lambda_i\}$ denote the eigenvalues of
$\A_{[0]}^2$ on the fiber ${\mathcal H}_m$.
By assumption, $\lambda_i \ge c^2$ for each $i$. Then for $t \ge 1$,
\begin{equation} \label{2.19}
\sum_i e^{- t^2 \lambda_i} = \sum_i e^{- (t^2 -1) \lambda_i}
e^{- \lambda_i} \le e^{- c^2 (t^2-1)} \sum_i e^{- \lambda_i}.
\end{equation}
Hence
\begin{equation} \label{2.20}
\left\| e^{- t^2 \A_{[0]}^2} \right\|_1 \le e^{- \: c^2 (t^2-1)}
\left\| e^{- \A_{[0]}^2} \right\|_1
\end{equation}
and it follows from the proof of
Lemma \ref{lem2} that on any compact subset of $M$,
\begin{equation} \label{2.21}
\Str \left(
e^{- \: \A_t^2} \right) = O \left( e^{- c^2 t^2/2} \right)
\end{equation}
and
\begin{equation} \label{2.22}
\Str \left( \frac{d \A_t}{dt}
e^{- \: \A_t^2} \right) = O \left( e^{- c^2 t^2/2} \right).
\end{equation}
In particular, the integrand in (\ref{2.17}) is integrable.

As in the proof of Lemma \ref{lem9},
\begin{equation} \label{2.23}
\frac{d}{dt} \Ch(\A_t) = - d \Str \left( \frac{d \A_t}{dt}
e^{- \: \A_t^2} \right).
\end{equation}
Using (\ref{2.21}) and (\ref{2.22}), we can integrate (\ref{2.23})
over $[1, \infty)$,
from which (\ref{2.18}) follows.
\end{proof}

\begin{lemma} \label{lem12}
Let $\E$ be a superconnection on a Hilbert bundle over $[0,1] \times M$.
Let $\E(s)$ be the superconnection on the restriction of the Hilbert
bundle to $\{s\} \times M$.
If $\E(s)_{[0]}$ is invertible for all $s \in [0,1]$
then
\begin{equation} \label{2.24}
\eta(\E(0),\infty) - \eta(\E(1), \infty) = - \int_0^1 \Ch(\E)
\end{equation}
in $\Omega^{odd}(M)/\Image(d)$.
In the case of a product Hilbert bundle,
meaning that it pulls back from a Hilbert bundle on $M$,
 (\ref{2.24}) becomes
\begin{equation} \label{2.25}
\eta(\E(0),\infty) - \eta(\E(1), \infty) = \eta(\E(0), \E(1)).
\end{equation}
\end{lemma}
\begin{proof}
Consider the superconnection $\F$ on $[1,\infty) \times [0,1] \times M$
given by
\begin{equation} \label{2.26}
dt \wedge \partial_t + \E_t.
\end{equation}
where $t$ is the coordinate on $[1, \infty)$. Given $L > 1$,
integrating $\Ch(\F)$ over $[1,L] \times [0,1]$ (c.f. (\ref{2.8})) gives
\begin{align} \label{2.27}
& \int_1^L \Str \left( \frac{d \E(0)_t}{dt} e^{- \E(0)_t^2} \right) dt -
\int_1^L \Str \left( \frac{d \E(1)_t}{dt} e^{- \E(1)_t^2} \right) dt = \\
& - \int_0^1 \Ch(\E) + \int_0^1 \Ch(\E_L). \notag
\end{align}
From (\ref{2.21}),
\begin{equation} \label{2.28}
\lim_{L \rightarrow \infty} \int_0^1 \Ch(\E_L) = 0.
\end{equation}
This proves (\ref{2.24}). Equation (\ref{2.25}) follows as in (\ref{2.12}).
\end{proof}

In Appendix \ref{app}
we prove a generalization of (\ref{2.25}) where we no longer
assume the invertibility of $\{ \E(s)_{[0]} \}_{s \in [0,1]}$.

\section{Differential $K$-theory} \label{sect3}

This section contains the main results of the paper.
In Subsection \ref{subsect3.1} we define the differential $K$-group
$\check{K}^0(M)$ in terms of superconnections on Hilbert
bundles over $M$. In Subsection \ref{subsect3.2}
we construct a map $q$
from $\check{K}^0(M)$
to the standard differential $K$-group $\check{K}^0_{stan}(M)$.
Subsection \ref{subsect3.3} makes the map more explicit when
the degree-$0$ part of the superconnection, $\A_{[0]}$, has
vector bundle kernel.  In Subsection \ref{subsect3.4} we show
that $q$ is an isomorphism, thereby proving Theorem \ref{thm1}.
Subsection \ref{mult} provides a multiplication on
$\check{K}^0(M)$.

\subsection{Definitions} \label{subsect3.1}

Let $M$ be a smooth manifold.

\begin{definition} \label{def4}
A cocycle for $\check{K}^0(M)$
is a triple $\left[ \mathcal{H}, \A, \omega \right]$
where
\begin{enumerate}
\item
$\mathcal{H}$ is a $\Z_2$-graded Hilbert bundle over $M$,
\item $\A$ is a superconnection on ${\mathcal H}$ and
\item $\omega \in \Omega^{odd}(M) /\Image(d)$.
\end{enumerate}
\end{definition}

\begin{definition} \label{def4.5}
Two cocycles $\left[ \mathcal{H}, \A, \omega \right]$ and
$\left[ \mathcal{H}', \A', \omega' \right]$ for $\check{K}^0(M)$
are isomorphic if
there is a smooth isometric isomorphism $i : {\mathcal H} \rightarrow
{\mathcal H}'$ so that $\left[ \mathcal{H}, \A, \omega \right] =
\left[ i^* \mathcal{H}', i^* \A', \omega' \right]$.
\end{definition}

\begin{definition} \label{def5}
The group $\check{K}^0(M)$ is the quotient of the free abelian
group generated by the isomorphism classes
of cocycles,
by the subgroup generated by
the following relations :
\begin{enumerate}
\item If $\left[ \mathcal{H}, \A, \omega \right]$ and
$\left[ \mathcal{H}', \A', \omega' \right]$ are cocycles then
\begin{equation} \label{3.2}
\left[ \mathcal{H}, \A, \omega \right]+
\left[ \mathcal{H}', \A', \omega' \right]=\left[ \mathcal{H} \oplus \mathcal{H}', \A \oplus \A', \omega+\omega' \right].
\end{equation}
\item If $\A_{[0]}$ is invertible then
\begin{equation} \label{3.3}
\left[ \mathcal{H}, \A, \omega \right] = \left[ 0, 0, \omega+\eta(\A, \infty) \right].
\end{equation}
\item Suppose that $\A_0$ and $\A_1$ are superconnections on
$\mathcal{H}$ such that
$\A_{0, [0]}- \A_{1, [0]} \in \Omega^0(M; op^0)$. Then
\begin{equation} \label{3.4}
\left[ \mathcal{H}, \A_0, \omega \right]=\left[ \mathcal{H}, \A_1, \omega+\eta(\A_0, \A_1) \right].
\end{equation}
\end{enumerate}
\end{definition}

\begin{example} \label{ex6}
In the setting of Example \ref{ex5}, given
$\omega \in \Omega^{odd}(M)/\Image(d)$,
the triple $\left[ {\mathcal H}, \A_{Bismut}, \omega \right]$ gives an
element of $\check{K}^0(M)$.  Compare with the ``geometric family''
cocycles of \cite{Bunke-Schick (2009)}.
\end{example}

It follows from the relations that there is a map
$\check{K}^0(M) \rightarrow \Omega^{even}(M)$ that sends a cocycle
$\left[ {\mathcal H}, \A, \omega \right]$ to $\Ch(\A) + d \omega$.

\begin{lemma} \label{extralem}
Let $\left[ \mathcal{H}, \A, \omega \right]$ be a cocycle for
$\check{K}^0(M)$.
Let $E$ be a finite dimensional Hermitian vector  bundle on $M$ with
compatible connection  $\nabla^E$. Put
$\widetilde{E} = E \oplus E$ with $\Z_2$-grading
$\begin{pmatrix}
1 & 0 \\
0 & -1
\end{pmatrix}$
and connection $\nabla^{\widetilde{E}} = \nabla^E \oplus \nabla^E$. Then
\begin{equation} \label{3.1}
\left[ \mathcal{H}, \A, \omega \right]= \left[ \mathcal{H}\oplus\widetilde{E}, \A \oplus \nabla^{\widetilde{E}}, \omega \right].
\end{equation}
\end{lemma}
\begin{proof}
From relation (1) of Definition \ref{def5}, it suffices to show that
$\left[ \widetilde{E}, \nabla^{\widetilde{E}}, 0 \right]$ vanishes
in $\check{K}^0(M)$. Define a superconnection $\B$ on $\widetilde{E}$ by
\begin{equation}
\B = \begin{pmatrix}
\nabla^E & I \\
I & \nabla^E
\end{pmatrix}.
\end{equation}
From relations (2) and (3) of Definition \ref{def5}, in $\check{K}^0(M)$
we have
\begin{align}
\left[ \widetilde{E}, \nabla^{\widetilde{E}}, 0 \right] = &
\left[ \widetilde{E}, \B, 0 \right] +
\eta \left( \nabla^{\widetilde{E}}, \B \right) \\
= & \left[ 0, 0, \eta \left( \nabla^{\widetilde{E}}, \B \right) +
\eta(\B, \infty) \right]. \notag
\end{align}

For $t \in [0,1]$, put
\begin{equation}
\A(t) = (1-t) \nabla^{\widetilde{E}} + t \B =
\begin{pmatrix}
\nabla^E & tI \\
tI & \nabla^E
\end{pmatrix}.
\end{equation}
Then
\begin{equation}
\A(t)^2 = \begin{pmatrix}
t^2 + (\nabla^E)^2 & 0 \\
0 & t^2 + (\nabla^E)^2
\end{pmatrix}
\end{equation}
and
\begin{align}
& \Str \left( \frac{d \A(t)}{dt}
e^{- \: \A^2(t)} \right) = \\
& \Str \left(
\begin{pmatrix}
0 & I \\
I & 0
\end{pmatrix}
\begin{pmatrix}
e^{-t^2 - (\nabla^E)^2} & 0 \\
0 & e^{-t^2 - (\nabla^E)^2}
\end{pmatrix}
\right) = 0. \notag
\end{align}
It follows from (\ref{2.9}) that
$\eta \left( \nabla^{\widetilde{E}}, \B \right) = 0$.
By a similar argument,
$\eta(\B, \infty) = 0$. This proves the lemma.
\end{proof}

\subsection{Map to the standard finite
dimensional version} \label{subsect3.2}

Define $\check{K}^0_{stan}(M)$ as in the previous subsection,
except using finite dimensional vector bundles
rather than Hilbert bundles, connections instead of superconnections,
removing relation (2)
and adding a stabilization relation
\begin{equation}
\left[ {\mathcal H}, \nabla, \omega \right] =
\left[ {\mathcal H} \oplus \widetilde{E} , \nabla \oplus
\nabla^{\widetilde{E}}, \omega \right]
\end{equation}
as in the conclusion of Lemma \ref{extralem}.

Then $\check{K}^0_{stan}(M)$ is isomorphic to the standard
differential $K$-theory group as defined, for example, in
\cite{Freed-Lott (2010)}.

Suppose now that $M$ is compact.
Given a cocycle
$c=\left[ \mathcal{H}, \A, \omega \right]$
for $\check{K}^0(M)$,
we construct an equivalent finite dimensional cocycle   as follows.

As in \cite[Section 9.5]{Berline-Getzler-Vergne (1996)},
one can find a finite dimensional
vector bundle $E$ on $M$ (in fact, a trivial one)
and a linear map $s : E \rightarrow {\mathcal H}^-$
so that $\A_{[0]}^+ + s \: : \:
{\mathcal H}^+ \oplus E \rightarrow {\mathcal H}^-$
is surjective. Let $\nabla^E$ be a connection on $E$. First,
from Lemma \ref{extralem}, the cocycle
$\left[ {\mathcal H}, \A, \omega \right]$ is equivalent to
$\left[ {\mathcal H} \oplus \widetilde{E}, \A \oplus \nabla^{\widetilde{E}},
\omega \right]$. Next, define $\Delta^+ \: : \:
{\mathcal H}^+ \oplus \widetilde{E}^+ \rightarrow
{\mathcal H}^- \oplus \widetilde{E}^-$ by
\begin{equation} \label{3.6}
\Delta^+ =
\begin{pmatrix}
0 & s \\
0 & 0
\end{pmatrix}
\end{equation}
and  define $\Delta^- \: : \:
{\mathcal H}^- \oplus \widetilde{E}^- \rightarrow {\mathcal H}^+ \oplus
\widetilde{E}^+$ by
\begin{equation} \label{3.7}
\Delta^- =
\begin{pmatrix}
0 & 0 \\
s^* & 0
\end{pmatrix}.
\end{equation}
Put
\begin{equation} \label{3.8}
\widehat{\A} = (\A \oplus \nabla^{\widetilde{E}}) +
\begin{pmatrix}
0 & \Delta^- \\
\Delta^+ & 0
\end{pmatrix},
\end{equation}
a superconnection on $\widehat{\mathcal H} = {\mathcal H} \oplus
\widetilde{E}$.
From relation (3) in Definition \ref{def5}, the cocycle
$\left[ \mathcal{H}, \A, \omega \right]$ is equivalent to
$\left[ \widehat{\mathcal{H}}, \widehat{\A}, \omega +
\eta (\A\oplus \nabla^{\widetilde{E}}, \widehat{\A} )\right]$.
As a map from ${\mathcal H}^+ \oplus \widetilde{E}^+$ to
${\mathcal H}^- \oplus \widetilde{E}^-$, we have
\begin{equation} \label{3.9}
\widehat{\A}_{[0]}^+ =
\begin{pmatrix}
\A_{[0]}^+ & s \\
0 & 0
\end{pmatrix}.
\end{equation}

Since $\A_{[0]}^+ + s$ is surjective, its kernel is a finite dimensional
vector bundle; c.f.
\cite[Section 9.5]{Berline-Getzler-Vergne (1996)}. Then
$\Ker( \widehat{\A}_{[0]}^+ ) = \Ker
( \A_{[0]}^+ + s )$
and
$\Ker( \widehat{\A}_{[0]}^- ) \cong
\Coker( \widehat{\A}_{[0]}^+ ) = \widetilde{E}^-$.
This shows that $\widehat{\A}_{[0]}$ has
($\Z_2$-graded) vector bundle kernel.

Let $P$ be orthogonal projection onto
$\Ker (\widehat{\A}_{[0]})$.
Put
\begin{equation} \label{3.10}
\widehat{\A}' = (I-P) \widehat{\A} (I-P) + P \widehat{\A}_{[1]} P.
\end{equation}
Then in $\check{K}^0(M)$,
$\left[ {\mathcal{H}}, {\A}, \omega \right]$ equals
\begin{align} \label{3.11}
& \left[ {\mathcal H} \oplus \widetilde{E}, \A \oplus \nabla^{\widetilde{E}},
\omega \right]
= \\
& \left[ \widehat{\mathcal{H}}, \widehat{\A}, \omega +
\eta(\A \oplus \nabla^{\widetilde{E}}, \widehat{\A}) \right] = \notag \\
&
\left[ \widehat{\mathcal{H}}, \widehat{\A}', \omega+
\eta(\A \oplus \nabla^{\widetilde{E}}, \widehat{\A}) +
\eta(\widehat{\A}, \widehat{\A}') \right] = \notag \\
 &
\left[ \widehat{\mathcal{H}}, \widehat{\A}', \omega+
\eta(\A \oplus \nabla^{\widetilde{E}}, \widehat{\A}') \right] = \notag \\
 &
\left[ (I-P)\widehat{\mathcal{H}},
(I-P) \widehat{\A} (I-P), 0 \right]+ \notag \\
& \left[\Ker (\widehat{\A}_{[0]}), P \widehat{\A}_{[1]} P, \omega+
\eta(\A \oplus \nabla^{\widetilde{E}}, \widehat{\A}') \right] = \notag \\
 &
\left[\Ker (\widehat{\A}_{[0]}),
P \widehat{\A}_{[1]} P, \right. \notag \\
& \left. \omega+\eta(\A \oplus \nabla^{\widetilde{E}},
\widehat{\A}') +\eta((I-P) \widehat{\A} (I-P), \infty) \right]. \notag
\end{align}
We define the last expression to be $q(c)$.

Thus given a cocycle $c$ for $\check{K}^0(M)$,  we
have constructed an equivalent cocycle $q(c)$ for $\check{K}^0(M)$, and
$q(c)$ is also a cocycle for $\check{K}^0_{stan}(M)$.

\begin{proposition} \label{prop3}
The class in $\check{K}^0_{stan}(M)$ represented by $q(c)$ is independent
of the choices of $E$, $s$ and $\nabla^E$.
\end{proposition}
\begin{proof}
We first claim that given $E$ and $s$, the class is independent of the
choice of $\nabla^E$. Let $\underline{\nabla}^E$ be another choice of
connection on $E$. Let $\widehat{\underline{\A}}$ and
$\widehat{\underline{\A}}^\prime$ be the corresponding superconnections
constructed using $\underline{\nabla}^E$ instead of $\nabla^E$.
Note that $\Ker(\widehat{\underline{\A}}_{[0]}) =
\Ker(\widehat{{\A}}_{[0]})$, so the projection operator $P$ doesn't change.
Then the change in the class of $q(c)$ in $\check{K}^0_{stan}(M)$ is
$[0,0,\delta \omega]$, where
\begin{align} \label{3.12}
\delta \omega = &
\eta(P \widehat{\underline{\A}}_{[1]} P,
P \widehat{{\A}}_{[1]} P) +
\eta({\A} \oplus \underline{\nabla}^{\widetilde{E}},
\widehat{\underline{\A}}') +\eta((I-P) \widehat{\underline{\A}} (I-P), \infty)-
\\
& \eta(\A \oplus \nabla^{\widetilde{E}},
\widehat{\A}')  - \eta((I-P) \widehat{\A} (I-P), \infty). \notag
\end{align}
Using Lemma \ref{lem10} and (\ref{2.25}), this equals
\begin{align} \label{3.13}
& \eta(P \widehat{\underline{\A}}_{[1]} P,
P \widehat{{\A}}_{[1]} P) +
\eta({\A} \oplus \underline{\nabla}^{\widetilde{E}},
\widehat{\underline{\A}}') + \\
& \eta((I-P) \widehat{\underline{\A}} (I-P),
(I-P) \widehat{\A} (I-P)) -
 \eta(\A \oplus \nabla^{\widetilde{E}}, \widehat{\A}') = \notag \\
& \eta({\A} \oplus \underline{\nabla}^{\widetilde{E}},
\widehat{\underline{\A}}') +\eta(\widehat{\underline{\A}}^\prime,
{\widehat{\A}}^\prime) -
 \eta(\A \oplus \nabla^{\widetilde{E}}, \widehat{\A}') = \notag \\
& \eta({\A} \oplus \underline{\nabla}^{\widetilde{E}},
{\A} \oplus {\nabla}^{\widetilde{E}}) =
\eta(\underline{\nabla}^{\widetilde{E}}, {\nabla}^{\widetilde{E}})  = 0. \notag
\end{align}
This proves the claim.

Now suppose that $E_1$, $s_1$, $E_2$ and $s_2$ are two different choices
for $E$ and $s$. Let $q_1(c)$ and $q_2(c)$ be the ensuing cocycles
for $\check{K}^0_{stan}(M)$.
Put $F = E_1 \oplus E_2$, with connection $\nabla^{E_1} \oplus
\nabla^{E_2}$. Put $S_1 = s_1 \oplus 0$ and
$S_2 = 0 \oplus s_2$, both being maps from $F$ to ${\mathcal H}^-$.
Let $C_1$ and $C_2$ be the ensuing cocycles
for $\check{K}^0(M)$. Then $q_1(c)$ is equivalent to $q(C_1)$, and
$q_2(c)$ is equivalent to $q(C_2)$.
For $t \in [0,1]$, put $S(t) = t S_2 + (1-t) S_1$. Then for all
$t \in [0,1]$, the map $\A_{[0]}^+ + S(t)$ is surjective from
${\mathcal H}^+ \oplus F$ to ${\mathcal H}^-$.
Hence we can reduce to the case when
$E_1 = E_2$, which we will again call $E$, but there are two maps
$s_1, s_2 : E \rightarrow {\mathcal H}^-$ that are joined by a
$1$-parameter family of maps $s(t) : E \rightarrow {\mathcal H}^-$,
so that
$\A_{[0]}^+ + s(t)$ is surjective for all $t \in [0,1]$.
If $q_1(c)$ is the cocycle for $\check{K}^0_{stan}(M)$ constructed using
$s(0)$, and $q_2(c)$ is the cocycle for $\check{K}^0_{stan}(M)$ constructed using
$s(1)$, then we want to show that $q_1(c) = q_2(c)$.

Define
$\widehat{\A}(t)$, $P(t)$,
and $\widehat{\A}^\prime(t)$ accordingly.
The family of Hilbert bundles
$\{ \widehat{\mathcal H}(t) \}_{t \in [0,1]}$
forms a Hilbert bundle ${\mathcal L}$ over $[0,1] \times M$.
There are subbundles ${\mathcal M}$ and ${\mathcal M}^\prime$ of ${\mathcal L}$
formed by $\{ P(t) \widehat{\mathcal H}(t) \}_{t \in [0,1]}$
and $\{(I-P(t)) \widehat{\mathcal H}(t) \}_{t \in [0,1]}$,
respectively.
The Hilbert bundle ${\mathcal L}$
acquires a superconnection
\begin{align} \label{3.14}
\widehat{\C}^\prime = &
(I-P(t))
\left( dt \wedge \partial_t + \widehat{\A}(t) \right) (I-P(t)) +  \\
& P(t) \left( dt \wedge \partial_t + \widehat{\A}_{[1]}(t) \right)
P(t). \notag
\end{align}
The finite dimensional bundle ${\mathcal M}$ acquires a
connection
\begin{equation} \label{3.15}
\B = P(t) \left( dt \wedge \partial_t +
\widehat{\A}_{[1]}(t) \right) P(t).
\end{equation}
There is a superconnection on ${\mathcal M}^\prime$ given by
\begin{equation} \label{3.16}
\widehat{\B}^{\prime} =
(I-P(t))
\left( dt \wedge \partial_t + \widehat{\A}(t) \right) (I-P(t)).
\end{equation}

In $\check{K}^0_{stan}(M)$, we have
\begin{align} \label{3.17}
&
\left[ \Ker(\widehat{\A}_{[0]}(1)), P(1) \widehat{\A}_{[1]}(1) P(1), 0 \right]
= \\
&
\left[ \Ker(\widehat{\A}_{[0]}(0)), P(0) \widehat{\A}_{[1]}(0) P(0), 0 \right]
+ \left[ 0,0,\int_0^1 \Ch(\B) \right], \notag
\end{align}
as can be seen from trivializing ${\mathcal M}$ with respect to
$[0,1]$ and then applying
(\ref{2.12}) and (\ref{3.4}).

From Lemma \ref{lem10},
\begin{align} \label{3.18}
\eta(\A \oplus \nabla^{\widetilde{E}}, \widehat{\A}^\prime (1)) -
\eta(\A \oplus \nabla^{\widetilde{E}}, \widehat{\A}^\prime (0)) = &
\eta(\widehat{\A}^\prime (0), \widehat{\A}^\prime (1)) \\
= & - \int_0^1 \Ch(\widehat{\C}^\prime). \notag
\end{align}
From Lemma \ref{lem12},
\begin{align} \label{3.19}
& \eta((I - P(1)) \widehat{\A}(1) (I - P(1)), \infty) = \\
& \eta((I - P(0)) \widehat{\A}(0) (I - P(0)), \infty) +
\int_0^1 \Ch(\B^{\prime}). \notag
\end{align}
Now
\begin{equation} \label{3.20}
\int_0^1 \Ch(\widehat{\C}^{\prime}) = \int_0^1 \Ch(\B) +
\int_0^1 \Ch(\B^{\prime}).
\end{equation}
Equations (\ref{3.17}), (\ref{3.18}), (\ref{3.19}) and (\ref{3.20}) imply that
$q_1(c) = q_2(c)$.
\end{proof}

\begin{proposition} \label{prop4}
$q$ passes to a map from $\check{K}^0(M)$ to $\check{K}^0_{stan}(M)$.
\end{proposition}
\begin{proof}
We have to show that $q$ vanishes on relations for $\check{K}^0(M)$.
This is evident for relations (1) and (2) in Definition \ref{def5}.
For relation (3), suppose that $\A_0$ and $\A_1$ are two superconnections
on ${\mathcal H}$ such that $\A_{0,[0]} - \A_{1,[0]} \in \Omega^0(M; op^0)$.
For $t \in [0,1]$, put $\A(t) = t \A_1 + (1-t) \A_0$. From
Corollary \ref{cor1},
we know that $\A(t) \in {\mathcal P}$.
Let ${\mathcal K}$ be the product Hilbert bundle $[0,1] \times {\mathcal H}$
over $[0,1] \times M$.
We can find a finite dimensional vector bundle $E$ on $[0,1] \times M$
(in fact a trivial one) and a map $s : E \rightarrow {\mathcal K}^-$
such that for each $t \in [0,1]$,
if $E(t) \rightarrow \{t\} \times M$ is the
restricted bundle and $s(t) : E(t) \rightarrow {\mathcal H}^-$ is the
restricted map then
$\A_{[0]}(t) + s(t) : {\mathcal H}^+
\oplus E(t) \rightarrow {\mathcal H}^-$ is surjective.
Define
$\widehat{\A}(t)$, $P(t)$,
and $\widehat{\A}^\prime(t)$ accordingly.
The family of Hilbert bundles
$\{ \widehat{\mathcal H}(t) \}_{t \in [0,1]}$
forms a Hilbert bundle ${\mathcal L}$ over $[0,1] \times M$.
There are subbundles ${\mathcal M}$ and ${\mathcal M}^\prime$ of ${\mathcal L}$
formed by $\{ P(t) \widehat{\mathcal H}(t) \}_{t \in [0,1]}$
and $\{(I-P(t)) \widehat{\mathcal H}(t) \}_{t \in [0,1]}$,
respectively. Define $\widehat{\C}^\prime$, $\B$ and
$\widehat{\B}^\prime$ as in (\ref{3.14}), (\ref{3.15}) and (\ref{3.16}).
As in the proof of Proposition \ref{prop3}, we conclude that
\begin{align} \label{3.21}
& \left[ \Ker(\widehat{A}_{[0]}(1)), P(1) \widehat{A}_{[1]} P(1), 0 \right] -
\left[ \Ker(\widehat{A}_{[0]}(0)), P(0) \widehat{A}_{[1]} P(0), 0 \right] = \\
& \left[ 0, 0, \int_0^1 \Ch(\B) \right] =
\left[ 0, 0, \int_0^1 \Ch(\widehat{\C}^\prime) \right] -
\left[ 0, 0, \int_0^1 \Ch({\B}^\prime) \right] = \notag \\
& \left[0, 0, - \eta(\A \oplus \nabla^{\widetilde{E}}, \widehat{\A}^\prime(1))
+ \eta(\A \oplus \nabla^{\widetilde{E}}, \widehat{\A}^\prime(0)) \right]
- \notag \\
& \left[ 0, 0, \eta((I - P(1)) \widehat{\A}(1) (I - P(1)), \infty) - \right.
\notag \\
& \left. \eta((I - P(0)) \widehat{\A}(0) (I - P(0)), \infty) \right].
\notag
\end{align}
The proposition follows.
\end{proof}

\subsection{Vector bundle kernel} \label{subsect3.3}

In this subsection we simplify the formula for $q(c)$
in the special case when
the family $\Ker(\A_{[0]})$ of vector spaces actually
form a vector bundle on $M$.

\begin{proposition} \label{prop5}
Suppose that $\Ker(\A_{[0]})$ is a ($\Z_2$-graded) vector bundle.
Let $Q$ denote orthogonal projection onto $\Ker(\A_{[0]})$.
Put
\begin{equation} \label{3.22}
\B = (I-Q) \A (I-Q) + Q \A_{[1]} Q.
\end{equation}
If an element $c \in \check{K}^0(M)$ is represented by
$c = [{\mathcal H}, \A, \omega]$ then $q(c) \in
\check{K}^0(M)_{stan}$ is represented by
\begin{equation} \label{3.23}
\left[ \Ker(\A_{[0]}), Q \A_{[1]} Q, \omega +
\eta(\A, \B) +
\eta((I-Q) \A (I-Q), \infty)
\right].
\end{equation}
\end{proposition}
\begin{proof}
Put $c = [{\mathcal H}, \A, \omega]$.
Then $c$ is equivalent to
$c^\prime = [{\mathcal H}, \B, \omega + \eta(\A, \B)]$, so
it suffices to look at $q(c^\prime)$.
In the construction of $q(c^\prime)$,
put $E = \Ker(\A_{[0]}^-)$, with connection
$\nabla^E = Q \A_{[1]}^- Q$. Let
$s : E \rightarrow {\mathcal H}^-$ be inclusion.
Define $\widehat{\B}$ as in (\ref{3.8}).
In terms of the orthogonal decompositions
\begin{equation} \label{3.24}
\widehat{\mathcal H}^+ = (I-Q){\mathcal H}^+ \oplus \Ker(\A_{[0]}^+) \oplus
\widetilde{E}^+
\end{equation}
and
\begin{equation} \label{3.25}
\widehat{\mathcal H}^- = (I-Q){\mathcal H}^- \oplus \Ker(\A_{[0]}^-) \oplus
\widetilde{E}^-,
\end{equation}
we can write
\begin{equation} \label{3.26}
\widehat{\B}_{[0]}^+ =
\begin{pmatrix}
(I - Q) \A_{[0]}^+ (I - Q) & 0 & 0 \\
0 & 0 & I \\
0 & 0 & 0
\end{pmatrix}
\end{equation}
and
\begin{equation} \label{3.27}
\widehat{\B}_{[0]}^- =
\begin{pmatrix}
(I - Q) \A_{[0]}^- (I - Q) & 0 & 0 \\
0 & 0 & 0 \\
0 & I & 0
\end{pmatrix}.
\end{equation}
Then
\begin{equation} \label{3.28}
\Ker(\widehat{\B}_{[0]}^+) \cong
0 \oplus \Ker(\A_{[0]}^+) \oplus 0
\end{equation}
and
\begin{equation} \label{3.29}
\Ker(\widehat{\B}_{[0]}^-) \cong 0 \oplus 0 \oplus \widetilde{E}^-.
\end{equation}

Let $P$ be projection onto $\Ker(\widehat{\B}_{[0]})$.
Put
\begin{equation} \label{3.30}
\widehat{\B}^\prime = (I-P) \widehat{\B} (I-P) +
P \widehat{\B}_{[1]} P.
\end{equation}
Then the cocycle
$q(c^\prime)$
is represented by
\begin{align} \label{3.31}
& \left[\Ker (\widehat{\B}_{[0]}),
P \widehat{\B}_{[1]} P, \right. \\
& \left. \omega+
\eta(\A, \B) + \eta(\B \oplus \nabla^{\widetilde{E}},
\widehat{\B}') +\eta((I-P) \widehat{\B} (I-P), \infty) \right]. \notag
\end{align}

As $P(Q \oplus 0)=P$, there is an orthogonal decomposition
\begin{equation}
\widehat{\mathcal H} = ((I - Q) \oplus 0) \widehat{\mathcal H} \oplus
((Q \oplus 0) - P) \widehat{\mathcal H} \oplus P \widehat{\mathcal H}.
\end{equation}
Put
\begin{equation}
{\mathcal H}^{\prime \prime} = (((Q \oplus 0) - P) \widehat{\mathcal H} =
\widetilde{E}^+ \oplus \Ker(\A_{[0]}^-),
\end{equation}
with the $\Z_2$-grading operator
$\begin{pmatrix}
1 & 0 \\
0 & -1
\end{pmatrix}$.
Define a
connection $\nabla^{{\mathcal H}^{\prime \prime}}$
by
$\nabla^{{\mathcal H}^{\prime \prime}} =
\begin{pmatrix}
\nabla^E & 0 \\
0 & \nabla^E
\end{pmatrix}$ and a superconnection $\E$ on
${\mathcal H}^{\prime \prime}$ by
$\E =
\begin{pmatrix}
\nabla^E & I \\
I & \nabla^E
\end{pmatrix}$.
Due to cancellations,
$\eta(\B \oplus \nabla^{\widetilde{E}},
\widehat{\B}') = \eta \left(
\nabla^{{\mathcal H}^{\prime \prime}}, \E \right)$.

Next,
\begin{equation} \label{3.32}
\eta((I-P) \widehat{\B} (I-P), \infty) = \eta((I-Q) \A (I-Q), \infty) +
\eta(\E, \infty).
\end{equation}
As in the proof of Lemma \ref{extralem},
\begin{equation} \label{straight}
\eta \left(
\nabla^{{\mathcal H}^{\prime \prime}}, \E \right)
+ \eta(\E, \infty) = 0.
\end{equation}
The proposition follows.
\end{proof}

\subsection{Proof of Theorem \ref{thm1}} \label{subsect3.4}

We now give the proof of Theorem \ref{thm1}, in the
following precise form.

\begin{proposition} \label{prop6}
The map $q \: : \: \check{K}^0(M) \rightarrow \check{K}^0_{stan}(M)$ is an
isomorphism.
\end{proposition}
\begin{proof}
Let
$e = \left[ F, \nabla^F, \omega \right]$
be a cocycle for $\check{K}^0_{stan}(M)$, where
$F$ is a finite dimensional $\Z_2$-graded Hermitian vector bundle
on $M$ and
$\nabla^F$ is a compatible connection on $F$.
We can also consider $e$ to be
a cocycle for $\check{K}^0(M)$.
Let $r(e)$ denote this cocycle for $\check{K}^0(M)$.
Then $r: 
\check{K}^0_{stan}(M) \rightarrow \check{K}^0(M)$ is
well defined, because of Definition \ref{def4.5} and Lemma \ref{extralem}.

Given a cocycle $c = \left[ {\mathcal H}, \A, \omega \right]$ for
$\check{K}^0(M)$, the cocycle
$r(q(c))$ amounts to considering $q(c)$ as a cocycle for
$\check{K}^0(M)$. Equation (\ref{3.11}) shows that
$r(q(c))$ is equivalent to $c$ in $\check{K}^0(M)$.

Given a cocycle $e$ for $\check{K}^0_{stan}(M)$,
applying Proposition \ref{prop5} to $r(e)$, with
$Q = \Id$, shows that $q(r(e))$ is equivalent to $e$ in
$\check{K}^0_{stan}(M)$. (Note that in this application, all of the
vector bundles in the proof of Proposition \ref{prop5} are finite
dimensional.) This proves the proposition.
\end{proof}

\begin{remark}
  The isomorphism in Proposition \ref{prop6} is compatible with
  the structures in differential cohomology theory.
  \end{remark} 

\subsection{Multiplicative structure} \label{mult}

Let $\left[ {\mathcal H}_1, \A_1, \omega_1 \right]$ and
$\left[ {\mathcal H}_2, \A_2, \omega_2 \right]$ be two
cocyles for $\check{K}^0(M)$. Let $\gamma_1$ be the $\Z_2$-grading
operator for ${\mathcal H}_1$. Put
\begin{align}
{\mathcal H} = & {\mathcal H}_1 \otimes {\mathcal H}_2, \\
\A = & (\A_1 \otimes I) + (\gamma_1 \otimes A_2), \notag \\
\omega = & (\Ch(\A_1) \wedge \omega_2) + (\omega_1 \wedge \Ch(\A_2))
+ (\omega_1 \wedge d\omega_2). \notag
\end{align}

\begin{proposition}
The map that sends $\left[ {\mathcal H}_1, \A_1, \omega_1 \right],
\left[ {\mathcal H}_2, \A_2, \omega_2 \right]$ to
$\left[ {\mathcal H}, \A, \omega \right]$ passes to a map
$m : \check{K}^0(M) \times \check{K}^0(M)  \rightarrow \check{K}^0(M) $.
\end{proposition}
\begin{proof}
First, the isomorphism class of $\left[ {\mathcal H}, \A, \omega \right]$
only depends on the isomorphism classes of
$\left[ {\mathcal H}_1, \A_1, \omega_1 \right]$ and
$\left[ {\mathcal H}_2, \A_2, \omega_2 \right]$.
We will check that the relations on the first factor
$\left[ {\mathcal H}_1, \A_1, \omega_1 \right]$ pass to relations on
$\left[ {\mathcal H}, \A, \omega \right]$. The argument for the second
factor is similar.

Relation (1) of Definition \ref{def5} is clearly compatible.
For relation (2) of Definition \ref{def5}, suppose that $\A_{1,[0]}$
is invertible.  Since
\begin{equation}
\A_{[0]}^2 = \A_{1,[0]}^2 + \A_{2,[0]}^2,
\end{equation}
it follows that $\A_{[0]}$ is invertible. It suffices to show that
\begin{align}
& \eta(\A, \infty) + (\Ch(\A_1) \wedge \omega_2) + (\omega_1 \wedge
\Ch(\A_2)) + (\omega_1 \wedge d\omega_2) = \\
& (\omega_1 + \eta(\A_1, \infty)) \wedge (\Ch(A_2) + d\omega_2)) \notag
\end{align}
in $\Omega^{odd}(M)/\Image(d)$ or, equivalently, that
\begin{equation} \label{suffices}
\eta(\A, \infty) =
\eta(\A_1, \infty) \wedge \Ch(A_2)
\end{equation}
in $\Omega^{odd}(M)/\Image(d)$.
Put
\begin{equation}
\widehat{\A}(t) = ((\A_1)_t \otimes \Id) + (\gamma_1 \otimes \A_2).
\end{equation}
Then
\begin{equation}
\Str \left( \frac{d\widehat{A}(t)}{dt} e^{- \widehat{\A}^2(t)\: } \right)
= \Str \left( \frac{d(\A_1)_t}{dt} e^{- (\A_1)_t^2\: } \right) \wedge
\Ch(\A_2).
\end{equation}
For $s \in [0,1]$, put
\begin{equation}
\B(s,t) = s \A_t + (1-s) \widehat{\A}(t).
\end{equation}
Then
\begin{equation}
\B(s,t)_{[0]} = t (\A_{1,[0]} \otimes \Id) + (st+1-s)
(\gamma_1 \otimes \A_{2,[0]})
\end{equation}
and
\begin{equation}
\B(s,t)_{[0]}^2 = (t^2 \A_{1,[0]}^2 \otimes \Id) + (st+1-s)^2
(\Id \otimes \A_{2,[0]}^2) \ge  t^2 \A_{1,[0]}^2 \otimes \Id.
\end{equation}

Using the strict positivity of $\A_{1,[0]}^2$, we can apply
a homotopy argument as in the proof of Lemma \ref{lem12}.
Since $\widehat{\A}(1) = \A$, it follows that
$\B(s,1) = \A$ for all $s \in [0,1]$. 
Then the analog of the term $\int_0^1 \Ch(\E)$ in Lemma \ref{lem12}
vanishes, so 
\begin{equation}
  \eta(\A, \infty) = \int_1^\infty \Tr_s \left(
\frac{d\widehat{\A}(t)}{dt} e^{- \: \widehat{\A}^2(t)} \right) \: dt =
\eta(\A_1, \infty) \wedge \Ch(A_2).
\end{equation}
in $\Omega^{odd}(M)/\Image(d)$.
This proves ( \ref{suffices}).

For relation (3) of Definition \ref{def5}, suppose that $\A_1^\prime$ is
another superconnection on ${\mathcal H}_1$ with
$\A_{1,[0]}^\prime - \A_{1,[0]} \in \Omega^0(M; op^0({\mathcal H}_1))$.
The new product
superconnection $\A^\prime$ on ${\mathcal H}$ satisfies
\begin{equation}
\A^\prime - \A = (\A_{1}^\prime - \A_{1}) \otimes \Id
\end{equation}
and hence
\begin{equation}
\A^\prime_{[0]} - \A_{[0]} = (\A_{1,[0]}^\prime - \A_{1,[0]}) \otimes \Id \in
op^0({\mathcal H}).
\end{equation}
It suffices to show that
\begin{align}
\Ch(\A_1) \wedge \omega_2 \: = & \: (\Ch(\A_1^\prime) \wedge \omega_2) +
(\eta(\A_1, \A_1^\prime) \wedge \Ch(\A_2)) + \\
& \: (\eta(\A_1, \A_1^\prime) \wedge d\omega_2) +
\eta(\A^\prime, \A) \notag
\end{align}
in $\Omega^{odd}(M)/\Image(d)$
or, equivalently,
\begin{equation} \label{equivalently}
\eta(\A, \A^\prime) = \eta(\A_1, \A_1^\prime) \wedge \Ch(\A_2).
\end{equation}
in $\Omega^{odd}(M)/\Image(d)$.
For $t \in [0,1]$, put
\begin{equation}
\A(t) = (1-t) \A + t \A^\prime
\end{equation}
and
\begin{equation}
\A_1(t) = (1-t) \A_1 + t \A_1^\prime.
\end{equation}
Then
\begin{equation}
\A^2(t) = (\A^2_1(t) \otimes \Id) + (\Id \otimes \A^2_2)
\end{equation}
and
\begin{equation}
\frac{d\A(t)}{dt} \: = \: \frac{d\A_1(t)}{dt} \otimes \Id,
\end{equation}
from which (\ref{equivalently}) follows.
\end{proof}

\begin{remark}
  It is straightforward to see that $m$ is unital and associative.
  That it is commutative can be shown using the graded isomorphism
  ${\mathcal H}_2 \otimes {\mathcal H}_1 \rightarrow
  {\mathcal H}_1 \otimes {\mathcal H}_2$ that sends
  $v_2 \otimes v_1$ to $(-1)^{\sign(v_1) \sign(v_2)} v_1 \otimes v_2$.
\end{remark}

\section{Pushforward} \label{sect4}

Given a fiber bundle $\pi : M \rightarrow B$ with even dimensional
compact fibers and a Riemannian
structure, in this section we construct a pushforward
$\pi_* : \check{K}^0_{fin}(M) \rightarrow \check{K}^0(B)$.
Here $\check{K}^0_{fin}(M)$ denotes the differential $K$-theory
constructed using finite dimensional Hermitian vector bundles with
compatible superconnections.
As $\check{K}^0_{fin}(M)$ is isomorphic to $\check{K}^0(M)$,
to define the pushforward on $\check{K}^0(M)$
it suffices to just define the pushforward on $\check{K}^0_{fin}(M)$.

In Subsection \ref{subsect4.1} we define the pushforward on
cocycles. In Subsection \ref{subsect4.2} we show that it passes
to a pushforward on the differential $K$-groups.  Subsection
\ref{subsect4.3} has the proof that $\pi_*$ coincides with the
analytic index of \cite{Freed-Lott (2010)}.

One could consider defining the pushforward directly on cocycles for
$\check{K}^0(M)$, instead of just on cocycles for
$\check{K}^0_{fin}(M)$. This may be possible but there are some
technical issues; see Remark \ref{added}.

\subsection{Pushforward on cocycles} \label{subsect4.1}

Let $\check{K}^0_{fin}(M)$ denote the group formed by only allowing
finite dimensional Hilbert bundles in the generators and relations of
Definition \ref{def5}. The proof of Proposition \ref{prop6}, when restricted to
finite dimensional Hilbert bundles, shows that $\check{K}^0_{fin}(M)$ is
isomorphic to $\check{K}^0_{stan}(M)$.
(The distinction between $\check{K}^0_{fin}(M)$ and
$\check{K}^0_{stan}(M)$ is that cocycles for the former involve
superconnections, whereas cocycles for the latter involve connections.)
From (\ref{3.11}), any cocycle for
$\check{K}^0(M)$ is equivalent to a cocycle for $\check{K}^0_{fin}(M)$.

Let $\pi \colon M \to B$ be a fiber bundle with even dimensional compact
fibers.
We assume that $\pi$ has a Riemannian structure in the sense of
\cite[Section 3.1]{Freed-Lott (2010)}. This means that we
are given an inner product $g^{T^V M}$ on the vertical tangent bundle
$T^VM = \Ker(d\pi)$ and a horizontal
distribution $T^H M$ on $M$. Let ${\mathcal D}^VM$ denote the
vertical density bundle on $M$.
We assume that $\pi$ is
${\mbox spin}^c$-oriented in the sense that the vertical tangent
bundle $T^V M$ on $M$ has a ${\mbox spin}^c$-structure, with
characteristic Hermitian line bundle $L^V M \rightarrow M$.
We also assume that $\pi$ has a differential ${\mbox spin}^c$-structure
in the sense of \cite[Section 3.1]{Freed-Lott (2010)}, meaning that
$L^V M$ is equipped with a unitary connection. Let $S^V M$ denote the
vertical spinor bundle on $M$. The connections on $T^V M$ and
$L^V M$ induce a connection $\widehat{\nabla}^{T^V M}$ on $S^V M$.

There is a degenerate Clifford module $C_0(M)$ generated by
$T^*M$, with a Clifford action $m$ on $\pi^*T^*B \otimes S^V M$
\cite[Section 10.2]{Berline-Getzler-Vergne (1996)}.
Note that we can identify   $C_0(M)$ and $\Lambda^*(T^*M)$ as vector
bundles on $M$.
Let $\B$ denote the Bismut superconnection acting on
$C^\infty(M; S^VM \otimes ({{\mathcal D}^V M})^{\frac12} )$.

Let $[E, \A, \omega]$ be a cocycle for
$\check{K}^0_{fin}(M)$.
We define the push-forward $\pi_*[E, \A, \omega]$ to be the cocycle
$[\mathcal{H}, \pi_*\A, \omega']$ where the definition of the terms is
as follows.  First, $\mathcal{H}$ is the Hilbert bundle whose fiber
over $b \in B$ is
\begin{equation} \label{4.1}
{\mathcal H}_b = L^2 \left( \pi^{-1}(b);
(E \otimes
S^VM \otimes ({{\mathcal D}^V M})^{\frac12}) \Big|_{\pi^{-1}(b)} \right) .
\end{equation}
The operator $D_b$ on ${\mathcal H}_b$ is the Dirac-type operator.
Then
\begin{equation} \label{4.2}
{\mathcal H}_b^\infty = C^\infty \left( \pi^{-1}(b);
(E \otimes
S^VM \otimes ({{\mathcal D}^V M})^{\frac12}) \Big|_{\pi^{-1}(b)} \right) .
\end{equation}
and
\begin{equation} \label{4.3}
C^\infty(B; {\mathcal H}^\infty) =
C^\infty(M; E \otimes S^VM \otimes ({{\mathcal D}^V M})^{\frac12}).
\end{equation}

We use the identification
\begin{align} \label{4.4}
& C^\infty(M; E \otimes S^VM \otimes ({{\mathcal D}^V M})^{\frac12}) = \\
& C^\infty(M; E)
\otimes_{C^\infty(M)}
C^\infty(M; S^VM \otimes ({{\mathcal D}^V M})^{\frac12}). \notag
\end{align}
Given $\xi \in C^\infty(M; E)$, write
\begin{equation} \label{4.5}
\A \xi = \sum_i \xi_i \otimes \omega_i \in
C^\infty(M; E \otimes \Lambda^*(T^*M)),
\end{equation} a locally finite sum on $M$,
where $\xi_i \in C^\infty(M; E)$ and
$\omega_i \in C^\infty(M; \Lambda^*(T^*M))$.
With $s \in C^\infty(M; S^VM \otimes ({{\mathcal D}^V M})^{\frac12})$,
let $m(\A \otimes 1)$ denote the operator that sends $\xi \otimes s$ to
\begin{equation} \label{4.6}
\sum_i \xi_i \otimes m(\omega_i) s \in
C^\infty(M; E \otimes S^VM \otimes
({{\mathcal D}^V M})^{\frac12} \otimes \pi^* T^*B).
\end{equation}

Acting on
$C^\infty(M; E \otimes S^VM \otimes ({{\mathcal D}^V M})^{\frac12})$ and
using (\ref{4.4}),
we define
\begin{equation} \label{4.7}
\pi_*\A = m(\A \otimes \Id) + \Id \otimes \B.
\end{equation}
This is well-defined since one can check, for example, that
if $\xi \in C^\infty(M; E)$, $s \in
C^\infty(M; S^VM \otimes ({{\mathcal D}^V M})^{\frac12})$ and
$f \in C^\infty(M)$ then
\begin{equation}
(\pi_*\A)(\xi f \otimes s) = (\pi_*\A)(\xi \otimes fs).
\end{equation}

Let ${\mathcal R}_t$ denote the rescaling operator on
superconnections, i.e. ${\mathcal R}_t \A = \A_t$ where the right-hand
side is defined in (\ref{2.16}).
Put
\begin{equation} \label{4.8}
  (\pi_*\A)_u = {\mathcal R}_u
  \left( m({\mathcal R}_{u^{-1}} \A \otimes \Id) + \Id \otimes \B \right).
\end{equation}
From \cite[Theorem 5.41]{Kahle (2011)}, the limit
$\lim_{u \to 0} \eta ((\pi_*\A)_u, \pi_*\A) \in \Omega^{odd}(B)/\Image(d)$
exists;
see also \cite[Appendix 1]{Bismut-Cheeger (1989)}.
Denote the limit by $\eta ((\pi_*\A)_0, \pi_*\A)$.
With the characteristic form $\Todd \left(
\widehat{\nabla}^{T^V M} \right)$ from \cite[Section 2.1]{Freed-Lott (2010)},
we put
\begin{equation} \label{4.9}
\omega' = \int_{M/B} \Todd \left(
\widehat{\nabla}^{T^V M} \right) \wedge \omega +
\eta ((\pi_*\A)_0, \pi_*\A).
\end{equation}

\begin{definition} \label{def6}
\begin{equation} \label{4.10}
\pi_* \left[ E, \A, \omega \right] =
\left[ {\mathcal H}, \pi_* \A, \omega^\prime \right].
\end{equation}
\end{definition}

\subsection{Pushforward on differential $K$-theory} \label{subsect4.2}

In this subsection we show that the pushforward $\pi_*$ is
well defined on $\check{K}^0_{fin}(M)$. We begin with a lemma.

\begin{lemma} \label{lem13}
Suppose that $\{\A(s)\}_{s \in [0,1]}$ is a smooth $1$-parameter
family of finite dimensional superconnections on $E$.
Then
\begin{equation} \label{4.11}
\lim_{u \rightarrow 0} \eta((\pi_* \A(1))_u, (\pi_* \A(0))_u)
= \int_{M/B} \Todd \left(
\widehat{\nabla}^{T^V M} \right) \wedge \eta(\A(1), \A(0)).
\end{equation}
\end{lemma}
\begin{proof}
Put $\E = ds \wedge \partial_s + \A(s)$ and
$\pi_* \E = ds \wedge \partial_s + \pi_* \A(s)$. Then
\begin{equation} \label{4.12}
\lim_{u \rightarrow 0} \eta((\pi_* \A(1))_u, (\pi_* \A(0))_u)
= \lim_{u \rightarrow 0} \int_0^1 \Ch((\pi_* \E)_u)
\end{equation}
and
\begin{align} \label{4.13}
& \int_{M/B} \Todd \left(
\widehat{\nabla}^{T^V M} \right) \wedge \eta(\A(1), \A(0)) = \\
& \int_0^1 \int_{M/B} \Todd \left(
\widehat{\nabla}^{T^V M} \right) \wedge \Ch(\E). \notag
\end{align}
From \cite[Theorem 5.33]{Kahle (2011)},
\begin{equation} \label{4.14}
\lim_{u \rightarrow 0} \Ch((\pi_* \E)_u) = \int_{M/B} \Todd \left(
\widehat{\nabla}^{T^V M} \right) \wedge \Ch(\E).
\end{equation}
uniformly on $[0,1] \times M$;
see also \cite[(2.127)]{Bismut-Gillet-Soule (1988)} and
\cite[Proposition 11.7]{Bismut-Lebeau (1991)}.
The lemma follows.
\end{proof}

\begin{proposition} \label{prop7}
The map $\pi_*$, as defined on cocycles for
$\check{K}^0_{fin}(M)$, passes
to a map $\pi_* : \check{K}^0_{fin}(M) \rightarrow \check{K}^0(B)$.
\end{proposition}
\begin{proof}
It is clear that relation (1), in Definition \ref{def5}
for $\check{K}^0_{fin}(M)$, passes through
$\pi_*$.

For relation (3), suppose that $\left[ E, \A, \omega \right]$ is a cocycle
for $\check{K}^0_{fin}(M)$ with $\A_{[0]}$ invertible. We must show that
$\pi_* \left[ E, \A, \omega \right]$ is equivalent to
\begin{equation} \label{4.15}
\pi_* \left[ 0,0, \omega + \eta(\A, \infty) \right] =
\left[ 0, 0, \int_{M/B} \Todd \left(
\widehat{\nabla}^{T^V M} \right) \wedge (\omega + \eta(\A, \infty)) \right].
\end{equation}
Equivalently, letting $F(\A) \in \check{K}^0(B)$ be the class represented by
the cocycle
\begin{equation} \label{4.16}
\left[ {\mathcal H}, \pi_* \A, \eta((\pi_* \A)_0, \pi_* \A) -
\int_{M/B} \Todd \left(
\widehat{\nabla}^{T^V M} \right) \wedge \eta(\A, \infty) \right],
\end{equation}
we must show that $F(\A)$ vanishes.

\begin{lemma} \label{lem14}
Suppose that $\{\A(s)\}_{s \in [0,1]}$ is a smooth $1$-parameter
family of finite dimensional superconnections on $M$, with
$\A(0)_{[0]}$ and $\A(1)_{[0]}$ invertible.
Then $F(\A(0)) = F(\A(1))$.
\end{lemma}
\begin{proof}
Using Lemmas \ref{lem10} and \ref{lem12}, one finds that
\begin{align} \label{4.17}
F(\A(1)) - F(\A(0)) = &
\left[ 0, 0, \lim_{u \rightarrow 0} \eta((\pi_* \A(1))_u, (\pi_* \A(0))_u)
- \right. \\
& \left. \int_{M/B} \Todd \left(
\widehat{\nabla}^{T^V M} \right) \wedge \eta(\A(1), \A(0)) \right]. \notag
\end{align}
The lemma now follows from Lemma \ref{lem13}.
\end{proof}

\begin{lemma} \label{lem15}
For $v > 0$, put $\A(v) = (v-1) \A_{[0]} + \A$. Then for all $b \in B$,
for sufficiently large $v$ the operator $(\pi_* \A(v))_{[0]}$ is invertible
on ${\mathcal H}_b$.
\end{lemma}
\begin{proof}
Without loss of generality, suppose that $B$ is a point.
Writing
\begin{equation} \label{4.18}
\A = \A_{[0]} +  \A_{[1]} + X,
\end{equation}
with $X \in C^\infty(M; \End(E) \otimes \Lambda^{\ge 2}(T^*M))$,
we have
\begin{equation} \label{4.19}
(\pi_* \A(v))_{[0]} = v (\A_{[0]} \otimes \Id) + D_{\A_{[1]}} +
m(X \otimes \Id),
\end{equation}
where $D_{\A_{[1]}}$ denotes the Dirac operator on $M$ coupled to the
connection $A_{[1]}$ on $E$.
Using the fact that $\A_{[0]}$ anticommutes with $\gamma_E$, while
$D_{\A_{[1]}}$ commutes with $\gamma_E$, it follows that
\begin{align} \label{4.20}
(\pi_* \A(v))_{[0]}^2 = & v^2 (\A_{[0]}^2 \otimes \Id) +
v m([\A_{[1]}, \A_{[0]}] \otimes \Id) + \\
& v m([\A_{[0]}, X] \otimes \Id) +
\left (D_{\A_{[1]}} + m(X \otimes \Id) \right)^2. \notag
\end{align}
Since $\A_{[0]}^2$ is strictly positive, if $v$ is sufficiently large
then $(\pi_* \A(v))_{[0]}^2$ is strictly positive.  This proves the lemma.
\end{proof}

By Lemmas \ref{lem14} and \ref{lem15}, without loss of generality we can assume
that $(\pi_* \A)_{[0]}$ is invertible.  Then the cocycle in (\ref{4.16}) is
equivalent to
\begin{equation} \label{4.21}
\left[ 0, 0,
\eta((\pi_* \A)_0, \pi_* \A) + \eta(\pi_* \A, \infty) -
\int_{M/B} \Todd \left(
\widehat{\nabla}^{T^V M} \right) \wedge \eta(\A, \infty) \right].
\end{equation}
By adiabatic techniques for eta forms, (\ref{4.21}) vanishes.
The proof of this is similar to the proofs in
\cite[Theorem 5.10 and Section 6]{Bunke-Ma (2004)}
and \cite{Liu (2015)}, which deal with
adiabatic limits in the more difficult case of a double fibration, i.e.
when $E$ is itself the pushforward of a vector bundle with connection.
We omit the details.

Finally, to show that relation
(3)
passes through $\pi_*$,
suppose that $\A_0$ and $\A_1$ are two superconnections on $E$.
Then
\begin{align} \label{4.22}
& \pi_* \left[ E , \A_1, \omega + \eta(\A_0, \A_1) \right]
- \pi_* \left[ E , \A_0, \omega \right] = \\
& \left[ {\mathcal H}, \pi_* \A_1, \int_{M/B} \Todd \left(
\widehat{\nabla}^{T^V M} \right) \wedge (\omega + \eta(\A_0, \A_1)) +
\eta((\pi_* \A_1)_0,\pi_* \A_1) \right] - \notag \\
& \left[ {\mathcal H}, \pi_* \A_0, \int_{M/B} \Todd \left(
\widehat{\nabla}^{T^V M} \right) \wedge \omega +
\eta((\pi_* \A_0)_0,\pi_* \A_0) \right] = \notag \\
& \left[ 0, 0,
\int_{M/B} \Todd \left(
\widehat{\nabla}^{T^V M} \right) \wedge \eta(\A_0, \A_1)
- \lim_{u \rightarrow 0} \eta((\pi_*\A_0)_u, (\pi_*\A_1)_u)
\right]. \notag
\end{align}
This vanishes from Lemma \ref{lem13}.
\end{proof}

\begin{remark} \label{added}
If we start with a cocycle $[E, \A, \omega]$ for $\check{K}^0(M)$,
i.e. an infinite dimensional cocycle, then we could consider
defining its pushforward $\left[ {\mathcal H}, \pi_* \A, \omega \right]$
as in Definition \ref{def6}. The definition of ${\mathcal H}$, as
(\ref{4.1}), still makes perfect sense. The formal definition of
$\pi_* \A$ is the same as (\ref{4.7}). However, there is the technical
point that we want $(\pi_* \A)_{[0]}$ to be $\theta$-summable
for all $\theta >0$. To analyze this
requirement, let us assume that $B$ is a point.
Then
\begin{equation}
\pi_* \A = \A_{[0]} + D_{\A_{[1]}} +
\sum_{i \ge 2} c( \A_{[i]}),
\end{equation}
where
$D_{\A_{[1]}}$ is the Dirac-type operator on
$C^\infty \left( M; E \otimes SM \otimes ({\mathcal D}M)^{\frac12} \right)$
as constructed using the connection $\A_{[1]}$ on $\E$,
and $c$
denotes the Clifford action on $SM$.
As
\begin{align}
(\pi_* \A)^2 = & \A_{[0]}^2 + c \left( \nabla \A_{[0]} \right) +
D_{\A_{[1]}}^2 + \\
& \sum_{i \ge 2}
c \left( [\A_{[0]}, \A_{[i]}] \right) +
\sum_{i \ge 2}
c \left( \nabla \A_{[i]} \right) +
\left( \sum_{i \ge 2} c( \A_{[i]}) \right)^2, \notag
\end{align}
in order for $\pi_* \A$ to be $\theta$-summable it is reasonable to
assume that $\nabla \A_{[0]}$ is $1^{st}$-order, and if $i \ge 2$ then
$\A_{[i]}$ and $\nabla \A_{[i]}$ are $0^{th}$-order.  (That is,
to require that
$\nabla \A_{[0]} \in C^\infty(M; T^* M \otimes op^1)$, and for $i \ge 2$ that
$\A_{[i]} \in C^\infty(M; \Lambda^i(T^*M) \otimes op^0)$ and
$\nabla \A_{[i]} \in C^\infty(M; T^*M \otimes \Lambda^i(T^*M) \otimes op^0)$.)

Returning to the case of general $B$,
even with such additional assumptions on $\A$, it remains to show that
$\eta((\pi_* \A)_0, \pi_* \A)$ is well defined, i.e. that the analog
of \cite[Theorem 5.41]{Kahle (2011)} holds for infinite dimensional bundles.
One can do all this for superconnections $\A$ arising from geometric
families in the sense of \cite{Bunke-Schick (2009)}, which gives reason
to believe that it can be done more generally.
\end{remark}

\subsection{Relation with the analytic index} \label{subsect4.3}

\begin{proposition} \label{prop8}
Under the isomorphisms
$\check{K}^0 \cong \check{K}^0_{fin} \cong \check{K}^0_{stan}$,
the map $\pi_* : \check{K}^0_{fin}(M) \rightarrow \check{K}^0(B)$ coincides
with the analytic index $\Ind^{an} :
\check{K}^0_{stan}(M) \rightarrow \check{K}^0_{stan}(B)$ of
\cite[Definition 3.12]{Freed-Lott (2010)}.
\end{proposition}
\begin{proof}
Suppose first that $\Ker((\pi_*\A)_{[0]})$ is a ($\Z_2$-graded) vector bundle
on $B$. Let $Q$ be orthogonal projection onto $\Ker((\pi_*\A)_{[0]})$.
For $T > 0$, put
\begin{equation} \label{4.23}
\E_T = (I-Q) (\pi_* \A)_T (I-Q) + Q (\pi_* \A)_{[1]} Q.
\end{equation}
By Proposition \ref{prop5} and Definition \ref{def6},
\begin{align} \label{4.24}
\pi_* \left[ E, \A, \omega \right] =
& \left[ \Ker(\A_{[0]}), Q (\pi_* \A)_{[1]} Q,
\int_{M/B} \Todd \left(
\widehat{\nabla}^{T^V M} \right) \wedge \omega + \right. \\
& \left. \eta( (\pi_* \A)_0, \pi_* \A) + \eta( \pi_* \A, \E_1) + \right.
\notag \\
& \left. \eta( (I-Q) \pi_* \A (I-Q), \infty) \right] \notag \\
= & \left[ \Ker(\A_{[0]}), Q (\pi_* \A)_{[1]} Q,
\int_{M/B} \Todd \left(
\widehat{\nabla}^{T^V M} \right) \wedge \omega + \right. \notag \\
& \left. \eta( (\pi_* \A)_0, (\pi_* \A)_T)) + \eta( (\pi_* \A)_T,\E_T) +
\right. \notag \\
& \left. \eta( (I-Q) (\pi_* \A)_T (I-Q), \infty) \right]. \notag
\end{align}
From Lemma \ref{lem8},
\begin{equation} \label{4.25}
\eta( (\pi_* \A)_0, (\pi_* \A)_T)) = \lim_{u \rightarrow 0}
\int_u^T
\Str \left( \frac{d (\pi_* \A)_t}{dt}
e^{- \: (\pi_* \A)_t^2} \right) dt.
\end{equation}
From \cite[Theorem 9.23]{Berline-Getzler-Vergne (1996)}, the limit
$\lim_{T \rightarrow \infty} \eta( (\pi_* \A)_0, (\pi_* \A)_T))$ exists;
it is called $\widetilde{\eta}$ in
\cite{Freed-Lott (2010)},
after the Bismut-Cheeger eta form
\cite{Bismut-Cheeger (1989)}.
Using the estimates in
\cite[Section 9.3]{Berline-Getzler-Vergne (1996)}, we have
$\lim_{T \rightarrow \infty} \eta( (\pi_* \A)_T,\E_T) = 0$ and
$\lim_{T \rightarrow \infty} \eta( (I-Q) (\pi_* \A)_T (I-Q), \infty) = 0$.
Thus
\begin{equation} \label{4.26}
\pi_* \left[ E, \A, \omega \right] =
\left[ \Ker(\A_{[0]}), Q (\pi_* \A)_{[1]} Q,
\int_{M/B} \Todd \left(
\widehat{\nabla}^{T^V M} \right) \wedge \omega +
\widetilde{\eta} \right],
\end{equation}
which is the same as \cite[Definition 3.12]{Freed-Lott (2010)}.
If $\Ker((\pi_*\A)_{[0]})$ is not a vector bundle then
we can effectively deform to that case, as in Subsection \ref{subsect3.2} and
\cite[Section 7.12]{Freed-Lott (2010)}.
\end{proof}

\section{Odd differential K-groups} \label{sect5}

In this section we indicate how the results of the previous
sections extend to the odd differential $K$-group $\check{K}^1(\cdot)$.
Some of the arguments are similar to those of the previous sections,
so we do not write them out in detail.  For this reason,
we label the results of this section as ``claims''.

We use the odd Chern characters of Quillen
\cite[Section 5]{Quillen (1985)}.
Let $M$ be a smooth manifold.
Let ${\mathcal H}$ be a Hilbert bundle over $M$ as in
Subsection \ref{subsect2.1}, except ungraded.
A superconnection $\A$ on ${\mathcal H}$ is defined as in
Definition \ref{def3}, except removing the
oddness condition on $\A_{[0]}$ and the parity
condition on $\A_{[i]}$. Note that the only grading on $\Omega^*(M; op^*)$
is the one coming from $\Omega^*(M)$.
Let $\sigma$ be a new formal odd variable
with $\sigma^2 = 1$. Put
\begin{equation} \label{5.1}
\A_\sigma = \sigma \A_{[0]} + \A_{[1]} + \sigma \A_{[2]} + \ldots,
\end{equation}
so that $\A_{\sigma}$ has odd total parity. Define
$\Tr_\sigma(A + B \sigma) = \Tr(B)$ and
\begin{equation} \label{5.2}
\Ch(\A) = \Tr_\sigma \left( e^{- \A_{\sigma}^2} \right) \in
\Omega^{odd}(M).
\end{equation}
Define the eta forms similarly to
(\ref{2.9}) and (\ref{2.17}), as elements of $\Omega^{even}(M)/\Image(d)$,
using $\Tr_\sigma$ instead of $\Tr_s$.

\begin{definition} \label{def7}
A cocycle for $\check{K}^1(M)$
is a triple $\left[ \mathcal{H}, \A, \omega \right]$
where
\begin{enumerate}
\item
$\mathcal{H}$ is an ungraded Hilbert bundle over $M$,
\item $\A$ is a superconnection on ${\mathcal H}$ and
\item $\omega \in \Omega^{even}(M) /\Image(d)$.
\end{enumerate}
\end{definition}

\begin{definition} \label{def8}
The group $\check{K}^1(M)$ is the quotient of the free abelian
group generated by the
isomorphism classes of
cocycles, by the subgroup generated by
the following relations :
\begin{enumerate}
\item If $\left[ \mathcal{H}, \A, \omega \right]$ and
$\left[ \mathcal{H}', \A', \omega' \right]$ are cocycles then
\begin{equation} \label{5.4}
\left[ \mathcal{H}, \A, \omega \right]+
\left[ \mathcal{H}', \A', \omega' \right]=\left[ \mathcal{H} \oplus \mathcal{H}', \A \oplus \A', \omega+\omega' \right].
\end{equation}
\item If $\A_{[0]}$ is invertible then
\begin{equation} \label{5.5}
\left[ \mathcal{H}, \A, \omega \right] = \left[ 0, 0, \omega+\eta(\A, \infty) \right].
\end{equation}
\item Suppose that $\A_0$ and $\A_1$ are superconnections on
$\mathcal{H}$ such that
$\A_{0, [0]}- \A_{1, [0]} \in \Omega^0(M; op^0)$. Then
\begin{equation} \label{5.6}
\left[ \mathcal{H}, \A_0, \omega \right]=\left[ \mathcal{H}, \A_1, \omega+\eta(\A_0, \A_1) \right].
\end{equation}
\end{enumerate}
\end{definition}

It follows from the definitions that there is a map $\check{K}^1(M) \rightarrow
\Omega^{odd}(M)$ that sends a cocycle $[{\mathcal H}, \A, \omega]$ to
$\Ch(\A) + d\omega$.

Let $\pi : M \rightarrow B$ be a fiber bundle with
odd dimensional compact fibers.  As in Section \ref{sect4},
we assume that $\pi$ has a Riemannian structure and a
differential $spin^c$-structure.  Given a cocycle
$[E, \A,\omega]$ for $\check{K}^0_{fin}(M)$, we define the
pushforward $\pi_*[E, \A, \omega]  =
[{\mathcal H}, \pi_* \A, \omega^\prime]$ as in
Definition \ref{def6}, where ${\mathcal H}$ is now ungraded.

\begin{claim}
The map $\pi_*$ on cocycles passes to a
map $\pi_* : \check{K}^0_{fin}(M) \rightarrow \check{K}^1(B)$.
\end{claim}
\begin{proof}
The proof is similar to that of
Proposition \ref{prop7}.
\end{proof}

Let $p_1 : S^1 \times M \rightarrow S^1$ and
$p : S^1 \times M \rightarrow M$ be the projection maps.
We define a suspension map $p^!$ on cocycles for
$\check{K}^1(M)$ as follows.  Consider the product bundle
$q : S^1 \times S^1 \rightarrow S^1$. Give it a product Riemannian
structure with circle fibers of constant length.
Let $(V, \nabla^V)$ be the
Hermitian line
bundle on $S^1 \times S^1$ with connection of constant curvature,
whose restriction to $q^{-1}(e^{i \theta})$
is the flat bundle on $S^1$ with holonomy $e^{i \theta}$.
The  sections of $V$ form a Hilbert bundle ${\mathcal K}$ over
$S^1$, on whose sections the
Bismut superconnection acts.
Trivializing $V$ over $(0, 2 \pi) \times S^1$ and giving coordinates
$(\theta, \phi)$ to the latter,
we can write this Bismut
superconnection on ${\mathcal K} \Big|_{(0, 2 \pi)}$
in the form $(\partial_{\phi} - \frac{i}{2\pi} \theta) +
d \theta \wedge \partial_\theta$.
Let $[{\mathcal H}, \A, \omega]$ be a cocycle for
$\check{K}^1(M)$.
Put
\begin{align} \label{5.7}
{\mathcal H}^\prime = & p^*{\mathcal H} \otimes
\left( p_1^* {\mathcal K} \oplus p_1^* {\mathcal K} \right), \\
\gamma_{{\mathcal H}^\prime} = &
\begin{pmatrix}
I & 0 \\
0 & -I
\end{pmatrix}, \notag \\
\sigma = &
\begin{pmatrix}
0 & I \\
I & 0
\end{pmatrix}, \notag \\
\A^\prime = &
\begin{pmatrix}
d \theta \wedge \partial_\theta & \partial_{\phi} - \frac{i}{2\pi}
\theta \\
- \partial_{\phi} + \frac{i}{2\pi} \theta
 & d \theta \wedge \partial_\theta
\end{pmatrix} + \A_{\sigma}, \notag \\
\omega^\prime = & p_1^*(d\theta) \wedge p^* \omega. \notag
\end{align}
Then $p^! [{\mathcal H}, \A, \omega] =
[{\mathcal H}^\prime, \A^\prime, \omega^\prime]$, a cocycle for
$\check{K}^0(S^1 \times M)$.

\begin{claim}
The map $p^!$ on cocycles passes to a map
$p^! : \check{K}^1(M) \rightarrow
\check{K}^0(S^1 \times M)$.
\end{claim}

Let $K^1(M)$ denote the group defined by similar generators and relations
as our definition of $\check{K}^1(M)$, except leaving out the
differential form components and the parts of the superconnection beyond
$\A_{[0]}$. As in our constructions on differential K-groups,
there are maps
$p_* : K^0(S^1 \times M) \rightarrow K^1(M)$ and
$p^! : K^1(M) \rightarrow K^0(S^1 \times M)$.
As in Proposition \ref{prop6}, $K^0(\cdot) \cong K^0_{stan}(\cdot)$ and so
\begin{align} \label{5.8}
& K^0(S^1 \times M) \cong
K^0_{stan}(S^1 \times M) \cong
K^0_{stan}(M) \oplus K^1_{stan}(M) \cong \\
& K^0(M) \oplus K^1_{stan}(M). \notag
\end{align}
Choosing a point $\star \in S^1$ and letting $i : \{ \star \} \times M
\rightarrow S^1 \times M$ be inclusion, it follows that
the group $K^0(S^1 \times M)$ is a direct sum of $K^0(M)$ and
$\Ker(i^* : K^0(S^1 \times M) \rightarrow K^0(M))$.

\begin{claim} \label{claim3}
The composite map $p_* \circ p^!$ is the identity on
$K^1(M)$, and $p^! \circ p_*$ is projection from
$K^0(S^1 \times M)$ onto
$\Ker(i^* : K^0(S^1 \times M) \rightarrow K^0(M))$.
\end{claim}

\begin{claim}
The group $K^1(M)$ is isomorphic to the standard $K$-group
$K^1_{stan}(M)$.
\end{claim}
\begin{proof}
Using Claim \ref{claim3},
\begin{align} \label{5.9}
& K^1(M) \cong
\Ker(i^* : K^0(S^1 \times M) \rightarrow K^0(M)) \cong \\
& \Ker(i^* : K^0_{stan}(S^1 \times M) \rightarrow K^0_{stan}(M))
\cong K^1_{stan}(M). \notag
\end{align}
The claim follows.
\end{proof}

Let $\Omega^{even}(M)_K$ denote the union of affine subspaces of closed forms
whose de Rham cohomology class lies in the image of
$\Ch : K^0(M) \rightarrow \Omega^{even}(M)$.

\begin{claim}
With our definition of $\check{K}^1(M)$,
there is a short exact sequence
\begin{equation} \label{5.10}
0 \rightarrow \frac{\Omega^{even}(M)}{\Omega^{even}(M)_K} \rightarrow
\check{K}^1(M) \rightarrow K^1(M) \rightarrow 0.
\end{equation}
\end{claim}

\begin{claim}
The differential $K$-group $\check{K}^1(M)$ is isomorphic to the
standard differential $K$-group $\check{K}^1_{stan}(M)$ as defined in
\cite[Section 9]{Freed-Lott (2010)} and \cite{Tradler-Wilson-Zeinalian (2013)}.
\end{claim}
\begin{proof}
There is a short exact sequence
\begin{equation} \label{5.11}
0 \rightarrow \frac{\Omega^{even}(M)}{\Omega^{even}(M)_K} \rightarrow
\check{K}^1_{stan}(M) \rightarrow K^1_{stan}(M) \rightarrow 0.
\end{equation}
Also, using the desuspension map
$D : \check{K}^0_{stan}(S^1 \times M) \rightarrow
\check{K}^1_{stan}(M)$ from
\cite[(9.21)]{Freed-Lott (2010)},
there is a map
\begin{equation} \label{5.12}
\check{K}^1(M) \stackrel{p^!}{\rightarrow} \check{K}^0(S^1 \times M)
\cong \check{K}^0_{stan}(S^1 \times M) \stackrel{D}{\rightarrow}
\check{K}^1_{stan}(M).
\end{equation}
The claim follows from applying the $5$-lemma to the diagram
\begin{equation} \label{5.13}
\begin{array}{ccccccccc}
0 & \rightarrow & \frac{\Omega^{even}(M)}{\Omega^{even}(M)_K} & \rightarrow
& \check{K}^1(M) & \rightarrow & K^1(M) & \rightarrow & 0 \\
& & \downarrow && \downarrow && \downarrow &&  \\
0 & \rightarrow & \frac{\Omega^{even}(M)}{\Omega^{even}(M)_K} & \rightarrow
& \check{K}^1_{stan}(M) & \rightarrow & K^1_{stan}(M) & \rightarrow & 0,
\end{array}
\end{equation}
where the rows are exact and the outer vertical arrows are
isomorphisms.
\end{proof}

\begin{remark} One can define multiplications
$\check{K}^0(M) \times \check{K}^1(M) \rightarrow \check{K}^1(M)$
and
$\check{K}^1(M) \times \check{K}^1(M) \rightarrow \check{K}^0(M)$
in analogy to the multiplication map of Section \ref{mult}.
\end{remark}

\begin{remark}
Definition 9.15 of \cite{Freed-Lott (2010)} is missing a homotopy
relation; we thank Scott Wilson for pointing this out.
Relation (2) on \cite[p. 955]{Freed-Lott (2010)} should
be replaced by a statement that two homotopy equivalent
unitary automorphisms are equivalent.  Then the corresponding relation on
\cite[p. 957]{Freed-Lott (2010)} will involve the transgressing
form $CS$ of \cite[p. 956]{Freed-Lott (2010)} when applied to
homotopy equivalent automorphisms;
c.f. \cite[Appendix A]{Tradler-Wilson-Zeinalian (2013)}.
\end{remark}

\section{Twisted differential $K$-theory} \label{twisted}

In this section we give the basic definitions for a Hilbert bundle
model of twisted differential $K$-theory, i.e. when differential
$K$-theory is twisted by an element of $\HH^3(M; \Z)$.
Subsection \ref{t1} has a review of abelian gerbes.  Subsection
\ref{t2} discusses connective structures on gerbes. Subsection
\ref{t3} recalls twisted de Rham cohomology.  In Subsection
\ref{t4} we define superconnections on projective Hilbert bundles.
Finally, Subsection \ref{psubsect3.1} contains the definition of
twisted differential $K$-theory.

\subsection{Gerbes} \label{t1}

We first give a summary of facts about abelian gerbes,
referring the reader \cite[Chapters 4 and 5]{Brylinski},
\cite{Hitchin (2001)} and \cite{Murray-Stevenson (2000)}
for more details.
We will describe abelian gerbes in terms of their descent data.

Let $M$ be a smooth manifold.
Let $\{U_\alpha\}_{\alpha \in \Lambda}$ be an open cover of $M$.
Put $U_{\alpha \beta} = U_\alpha \cap U_{\beta}$,
$U_{\alpha \beta \gamma} = U_\alpha \cap U_{\beta} \cap U_{\gamma}$, etc.
Any statement about $U_{\alpha \beta}$ will only refer to the case when
$U_{\alpha \beta} \neq \emptyset$, and similarly for
$U_{\alpha \beta \gamma}$ and $U_{\alpha \beta \gamma \delta}$.

A unitary gerbe $\mathcal{L}$ on $M$ is described by the following data:
\begin{itemize}
\item A Hermitian line bundle ${\mathcal L}_{\alpha \beta}$ over
each $U_{\alpha\beta}$ and
\item An isometric isomorphism
$\mu_{\alpha\beta\gamma}\colon \mathcal{L}_{\alpha\beta}
\Big|_{U_{\alpha \beta \gamma}} \otimes\mathcal{L}_{\beta\gamma}
\Big|_{U_{\alpha \beta \gamma}} \to
\mathcal{L}_{\alpha\gamma} \Big|_{U_{\alpha \beta \gamma}}$
over each $U_{\alpha \beta \gamma}$ such that
\item Over each $U_{\alpha\beta\gamma\delta}$, the following diagram commutes :
\begin{equation}
\begin{CD}
\mathcal{L}_{\alpha\beta}\otimes \mathcal{L}_{\beta\gamma}\otimes\mathcal{L}_{\gamma\delta}
@>\mu_{\alpha\beta\gamma}\otimes \id>>
\mathcal{L}_{\alpha\gamma}\otimes\mathcal{L}_{\gamma\delta}\\
@V \id\otimes\mu_{\beta\gamma\delta}VV  @VV\mu_{\alpha\gamma\delta}V\\
\mathcal{L}_{\alpha\beta}\otimes \mathcal{L}_{\beta\delta} @>\ \ \mu_{\alpha\beta\delta}\ \ >>\mathcal{L}_{\alpha\delta}.
\end{CD}
\end{equation}
\end{itemize}

If the cover $\{U_{\alpha}\}_{\alpha \in \Lambda}$ is good then the bundles
$\mathcal{L}_{\alpha \beta}$ are trivializable.
After choices of trivializations,
the collection $(\mu_{\alpha\beta\gamma})$ can be viewed  as a \v{C}ech $2$-cocycle with coefficents in the sheaf $\underline{\mathbb{T}}$ of smooth functions with values in $\mathbb{T}= \{ z \in \mathbb{C} \mid |z|=1\}$.
This defines a cohomology class  $[\mu]\in \HH^2(M; \underline{\mathbb{T}}) \cong \HH^3(M, \Z)$.

Given the open cover, there is an evident notion of two gerbes being
isomorphic. There is also an evident notion of the tensor product of
two gerbes.  A {\em trivial} unitary gerbe has defining line bundle
${\mathcal L}_{\alpha \beta}$ on $U_{\alpha \beta}$ given by
$\widehat{\mathcal L}_\alpha \otimes
\widehat{\mathcal L}_\beta^{-1}$, where
$\widehat{\mathcal L}_\alpha$ is a Hermitian line bundle on $U_\alpha$.
        {\em Stable isomorphism} of unitary gerbes is the equivalence relation
        generated by isomorphism along with saying that
        if ${\mathcal L}$ is a unitary gerbe, and
        ${\mathcal L}^\prime$ is a trivial unitary gerbe, then
        ${\mathcal L}$ is equivalent to ${\mathcal L} \otimes
        {\mathcal L}^\prime$.
        The class $[\mu]$ determines the stable isomorphism class of the
        gerbe. 

\subsection{Connective structures on gerbes} \label{t2}

\begin{definition}
A connective structure on $\mathcal{L}$ is given by a collection $(\nabla_{\alpha\beta})$ of Hermitian connections on $(\mathcal{L}_{\alpha\beta})$ such that
on $U_{\alpha \beta \gamma}$, we have
\begin{equation}
\mu_{\alpha\beta\gamma}^* \nabla_{\alpha\gamma} = \nabla_{\alpha\beta}\otimes \id + \id \otimes \nabla_{\beta\gamma}.
\end{equation}
\end{definition}

\begin{fact}
A connective structure exists on any gerbe.
\end{fact}

Let $\nabla=(\nabla_{\alpha\beta})$ be a connective structure on
$\mathcal{L}$. Put $\kappa_{\alpha\beta}=\nabla^2_{\alpha\beta}$,
the curvature
of the connection $\nabla_{\alpha \beta}$.

\begin{definition}
A curving of $\nabla$ is given
 by a collection  $K$ of $2$-forms $\kappa_\alpha\in \Omega^2(U_\alpha)$
such that on $U_{\alpha \beta}$, we have
\begin{equation}
\kappa_{\alpha\beta}=\kappa_{\alpha} - \kappa_{\beta}.
\end{equation}
\end{definition}

\begin{fact}
Given a gerbe with connective structure $\nabla$, there is a curving of $\nabla$.
\end{fact}

Given $\tau \in \Omega^2(M)$ and a curving $K=\{\kappa_\alpha\}$ on a gerbe $\mathcal{L}$ with connective structure $\nabla$, we can define
a new curving $K+\tau = \{\kappa_\alpha+\tau\}$. This defines a free transitive action of $\Omega^2(M)$ on the set of all curvings (on a given
connective structure).

There is a closed form
$c(K) \in \Omega^3(M)$
such that $c(K)|_{U_{\alpha}}= d\kappa_{\alpha}$. Note that $c(K+\tau)= c(K) + d \tau$.

If ${\mathcal L}$ is a trivial unitary gerbe with
${\mathcal L}_{\alpha \beta} = \widehat{\mathcal L}_{\alpha} \otimes
\widehat{\mathcal L}_{\beta}^{-1}$, and
$\widehat{\nabla}_\alpha$ is a Hermitian
connection on $\widehat{\mathcal L}_{\alpha}$, then
there is a connective structure $\nabla$ on ${\mathcal L}$
given by $\nabla_{\alpha \beta} =
\widehat{\nabla}_\alpha \otimes \widehat{\nabla}_\beta^{-1}$.
There is a curving $K$ of $\nabla$ coming from $\kappa_\alpha =
\nabla_\alpha^2$. It has $c(K) = 0$. 

\subsection{Twisted cohomology} \label{t3}

Let $H\in \Omega^3(M)$ be a closed $3$-form.   The periodic twisted de Rham complex is   the $\mathbb{Z}_2$-graded complex
$\left(\Omega^*(M), d_H \right)$
\begin{equation}
\ldots \overset{d_H}{\longrightarrow} \Omega^{even}(M) \overset{d_H}{\longrightarrow} \Omega^{odd}(M) \overset{d_H}{\longrightarrow}\Omega^{even}(M) \overset{d_H}{\longrightarrow}\ldots
\end{equation}
with the differential $d_H=d +H\wedge \cdot$.
Its cohomology is called the twisted de Rham cohomology.
If $H'=H+d\tau $
then there is an isomorphism of complexes   $I_\tau \colon \left(\Omega^*(M), d_{H} \right) \to \left(\Omega^*(M), d_{H'} \right)$ where
\begin{equation}\label{defI}
I_\tau (\xi) = e^{-\tau}\wedge \xi.
\end{equation}

\subsection{Projective Hilbert bundles} \label{t4}

\begin{definition}
An $\mathcal{L}$ -projective Hilbert bundle $\mathcal{H}$ is given by
\begin{itemize}
\item A Hilbert bundle
(in the sense of Subsection \ref{subsect2.1})
$\mathcal{H}_\alpha$ over each $U_\alpha$,
and
\item A collection of Hilbert bundle isomorphisms $\varphi_{\alpha\beta}: \mathcal{H}_\alpha\otimes \mathcal{L}_{\alpha\beta} \cong \mathcal{H}_\beta$ such that
\item On $U_{\alpha \beta \gamma}$, the following diagram commutes :

\begin{equation}
\begin{CD}
\mathcal{H}_\alpha\otimes \mathcal{L}_{\alpha\beta}\otimes \mathcal{L}_{\beta\gamma} @>\id \otimes \mu_{\alpha\beta\gamma}>>
\mathcal{H}_\alpha \otimes \mathcal{L}_{\alpha\gamma}\\
@V\varphi_{\alpha\beta}\otimes \id VV    @VV\varphi_{\alpha\gamma}V \\
\mathcal{H}_\beta\otimes \mathcal{L}_{\beta\gamma} @>\ \ \varphi_{\beta\gamma}\ \ >> \mathcal{H}_{\gamma}.
\end{CD}
\end{equation}
\end{itemize}
\end{definition}

\begin{example}
  To a unitary gerbe ${\mathcal L}$ on $M$, one can associate a
  projective space bundle on $M$ \cite[Section 4.1]{Brylinski}.
Using this bundle, there is a natural ${\mathcal L}$-projective Hilbert
bundle on $M$ formed by fermionic Fock spaces; c.f.
\cite{Mickelsson (2006)}.
\end{example}

Note that there is a canonical vector bundle isomorphism over $U_{\alpha \beta}$ :
\begin{equation}
 op^k(\mathcal{H}_\alpha \otimes \mathcal{L}_{\alpha\beta}) \cong  op^k(\mathcal{H}_\alpha).
\end{equation}
Then the bundle isomorphism $\varphi_{\alpha\beta}$ induces
an isomorphism of algebra bundles over $U_{\alpha\beta}$ :
\begin{equation}
\varphi_{\alpha\beta}^* \colon op^k(\mathcal{H}_\beta) \to  op^k(\mathcal{H}_\alpha).
\end{equation}
Over $U_{\alpha \beta \gamma}$, we have
\begin{equation}
\varphi_{\alpha\beta}^* \circ \varphi_{\beta\gamma}^* = \varphi_{\alpha\gamma}^*.
\end{equation}
Hence the collection $\{op^k(\mathcal{H}_\alpha)\}_{\alpha \in \Lambda}$
defines a bundle of algebras over $M$, which we denote by
$op^k(\mathcal{H})$. Similarly, there is
a bundle $\mathcal{L}^1(\mathcal{H})$ of trace ideals.
For every $\alpha \in \Lambda$, we can apply the fiberwise trace to
smooth sections
of $\mathcal{L}^1(\mathcal{H}_{\alpha})$, to obtain a map
\begin{equation}
\Tr_{\alpha} \colon
C^\infty(U_\alpha; \mathcal{L}^1(\mathcal{H}_{\alpha})) \to
C^{\infty}(U_{\alpha}).
\end{equation}
For $b \in C^\infty(U_{\beta}; \mathcal{L}^1(\mathcal{H}_{\beta}))$,
we have $\Tr_{\alpha}(\varphi_{\alpha \beta}^*(b))= \Tr_{\beta}(b)$ on
$U_{\alpha \beta}$.
Hence we can define $\Tr \colon C^\infty(M; \mathcal{L}^1(\mathcal{H}))
\to C^{\infty}(M)$ by saying that
\begin{equation}
\Tr(a) \Big|_{U_{\alpha}}= \Tr_{\alpha} \left( a \Big|_{U_{\alpha}} \right).
\end{equation}
If the bundle $\mathcal{H}$ is $\Z_2$-graded then there is a
supertrace
\begin{equation}
\Str \colon C^\infty(M; \mathcal{L}^1(\mathcal{H})) \to C^{\infty}(M).
\end{equation}

Assume now that the gerbe $\mathcal{L}$ is equipped with a connective
structure $\nabla$.
A superconnection on $ \mathcal{H} $  is a choice of superconnection
$\mathbb{A}_\alpha$ on each
$\mathcal{H}_\alpha$ so that on $U_{\alpha \beta}$, we have
\begin{equation}\label{connectionon}
\varphi_{\alpha\beta}^* \mathbb{A}_\beta = \mathbb{A}_\alpha\otimes \id + \id \otimes \nabla_{\alpha\beta}.
\end{equation}

Let $K$ be a curving of $\nabla$. Put $H= c(K)$.
\begin{lemma} There is a
$\theta_{\mathbb{A}, K} \in \Omega^*(M, op^*(\mathcal{H}))$ such that
for all $\alpha \in \Lambda$, we have
\begin{equation}
\left(\theta_{\mathbb{A}, K}\right)
\Big|_{U_\alpha} = \mathbb{A}_\alpha^2+\kappa_\alpha.
\end{equation}
\end{lemma}
\begin{proof}
Equation \eqref{connectionon} implies that on $U_{\alpha \beta}$, we have
\begin{equation}
\varphi_{\alpha\beta}^* \mathbb{A}_\beta^2 = \mathbb{A}_\alpha^2 +  \kappa_{\alpha\beta}.
\end{equation}
Hence the collection
$\{\mathbb{A}_\alpha^2 + \kappa_\alpha\}_{\alpha \in \Lambda}$  satisfies the relations
\begin{equation}
\varphi_{\alpha\beta}^*(\mathbb{A}_\beta^2 + \kappa_\beta) = \mathbb{A}_\alpha^2 + \kappa_\alpha.
\end{equation}
The lemma follows.
\end{proof}

For $\xi \in \Omega^*(M, op^*(\mathcal{H}))$ we can define $[\mathbb{A}, \xi] \in \Omega^*(M, op^*(\mathcal{H}))$ by
\begin{equation}
[\mathbb{A}, \xi] \Big|_{U_\alpha} = \left[
\mathbb{A}_\alpha, \xi \Big|_{U_\alpha} \right].
\end{equation}
One can check that this is well defined.

\begin{lemma}
We have
$ [\mathbb{A}, \theta_{\mathbb{A}, K}] = H$.
\end{lemma}
\begin{proof}
On $U_\alpha$, we know that
$[\mathbb{A}_\alpha, \mathbb{A}_\alpha^2+\kappa_\alpha]= d\kappa_\alpha =
H \Big|_{U_\alpha}$. The lemma follows.
\end{proof}

\begin{definition}
The Chern character of $\A$ is given by
\begin{equation}
\Ch(\A, K) = \Str e^{- \: \theta_{\mathbb{A}, K}} \in \Omega^{even}(M).
\end{equation}
\end{definition}

\begin{lemma}
\begin{enumerate}
\item We have $d_H \Ch(\A, K)=0$.
\item Let $K'=K+\tau$ be another curving. Put $H'=c(K')=H+d \tau$. Then
\begin{equation}
I_\tau (\Ch(\A, K)) = \Ch(\A, K').
\end{equation}
\end{enumerate}
\end{lemma}
 \begin{proof}
(1). We have
\begin{align}
d \Str e^{- \: \theta_{\mathbb{A}, K}}= &
\Str [\mathbb{A}, e^{- \: \theta_{\mathbb{A}, K}}]=
- \Str [\mathbb{A}, \theta_{\mathbb{A}, K}] e^{- \: \theta_{\mathbb{A}, K}}
\\
= &- \Str H e^{- \: \theta_{\mathbb{A}, K}}= -H \Str e^{- \: \theta_{\mathbb{A}, K}}. \notag
\end{align}
Hence $d_H \Ch(\A, K)=0$.   \\
(2). As $\theta_{\mathbb{A}, K'} = \theta_{\mathbb{A}, K}+ \tau$, the lemma
follows.
 \end{proof}

The proofs of the next four lemmas are similar to those in
Subsection \ref{subsect2.2}.

\begin{lemma} %\label{lem8}
Let $\{\A(t)\}_{t \in [0,1]}$ and $\{\widehat{\A}(t)\}_{t \in [0,1]}$
be two smooth $1$-parameter families of superconnections on ${\mathcal H}$
with $\A(0) = \widehat{\A}(0)$ and $\A(1) = \widehat{\A}(1)$.
Suppose that the two $1$-parameter families are homotopic
relative to the endpoints, in sense that there is
a smooth $2$-parameter family of superconnections
$\{\widetilde{\A}(s,t)\}_{s,t \in [0,1]}$ on ${\mathcal H}$
with $\widetilde{\A}(0,t) = \A(t)$, $\widetilde{\A}(1,t) =
\widehat{\A}(t)$, $\widetilde{\A}(s,0) = \A(0) = \widehat{\A}(0)$ and
$\widetilde{\A}(s,1) = \A(1) = \widehat{\A}(1)$.

Then
\begin{equation} \label{p2.5}
\int_0^1 \Str \left( \frac{d \A(t)}{dt}
e^{- \: \theta_{\mathbb{A}(t), K}} \right) dt =
\int_0^1 \Str \left( \frac{d \widehat{\A}(t)}{dt}
e^{- \: \theta_{\widehat{\A}(t), K}} \right) dt
\end{equation}
in $\Omega^{odd}(M)/\Image(d_H)$.
\end{lemma}

Let $\A_0$ and $\A_1$ be two superconnections on ${\mathcal H}$ such that
$\A_{0,[0]}- \A_{1,[0]} \in \Omega^0(M; op^0({\mathcal H}))$.
For $t \in [0,1]$, put $\A(t) = (1-t)\A_0 + t\A_1$.
Define $\eta(\A_0, \A_1) \in \Omega^{odd}(M)/ \Image(d_H)$ by
\begin{equation} \label{p2.9}
\eta(\A_0, \A_1) = \int_0^1 \Str \left( \frac{d \A(t)}{dt}
e^{- \: \theta_{\mathbb{A}(t), K}} \right) dt.
\end{equation}

\begin{lemma} %\label{lem9}
\begin{equation} \label{p2.10}
\Ch(\A_1)- \Ch(\A_0) = - d_H \eta(\A_0, \A_1).
\end{equation}
\end{lemma}

\begin{lemma} %\label{lem10}
Let $\A_0$, $\A_1$ and $\A_2$ be three superconnections on ${\mathcal H}$
such that $\A_{0,[0]} - \A_{1,[0]} \in \Omega^0(M; op^0({\mathcal H}))$ and
$\A_{1,[0]} - \A_{2,[0]} \in \Omega^0(M; op^0({\mathcal H}))$. Then
\begin{equation} \label{p2.14}
\eta(\A_0, \A_1) +\eta(\A_1, \A_2) = \eta(\A_0, \A_2).
\end{equation}
\end{lemma}

\begin{lemma} %\label{lem11}
Suppose that there is some $c > 0$ so that $\A_{[0]}^2 \ge c^2 \Id$
fiberwise on ${\mathcal H}$.
Put
\begin{equation} %\label{2.17}
\eta(\A, \infty) = \int_1^\infty \Str \left( \frac{d \A_t}{dt}
e^{- \: \theta_{\mathbb{A}_t, K}} \right) dt.
\end{equation}
Then
\begin{equation} \label{p2.18}
\Ch(\A) = d_H \eta(\A, \infty)
\end{equation}
\end{lemma}

\subsection{Definition of twisted differential $K$-theory} \label{psubsect3.1}

Let $M$ be a smooth manifold. Let $\mathcal{L}$ be a gerbe on $M$ with a
connective structure $\nabla$.

Let $K$ be a curving of $\nabla$. Put $H=c(K)$.

\begin{definition} \label{pdef4}
  A cocycle for
$\check{K}^0_{\mathcal L}(M)$ 
is a triple $\left[ \mathcal{H}, \A, \omega \right]$
where
\begin{enumerate}
\item
  $\mathcal{H}$ is a $\Z_2$-graded
  ${\mathcal L}$-projective
  Hilbert bundle over $M$,
\item $\A$ is a superconnection on ${\mathcal H}$ and
\item $\omega \in \Omega^{odd}(M) /\Image(d_H)$.
\end{enumerate}
\end{definition}

\begin{definition} \label{pdef5}
  The twisted differential $K$-theory group
  $\check{K}^0_{\mathcal L}(M)$ 
  is the quotient of the free abelian
group generated by the isomorphism classes of
cocycles, by the subgroup generated by
the following relations :
\begin{enumerate}
\item If $\left[ \mathcal{H}, \A, \omega \right]$ and
$\left[ \mathcal{H}', \A', \omega' \right]$ are cocycles then
\begin{equation} \label{p3.2}
\left[ \mathcal{H}, \A, \omega \right]+
\left[ \mathcal{H}', \A', \omega' \right]=\left[ \mathcal{H} \oplus \mathcal{H}', \A \oplus \A', \omega+\omega' \right].
\end{equation}
\item If $\A_{[0]}$ is invertible then
\begin{equation} \label{p3.3}
\left[ \mathcal{H}, \A, \omega \right] = \left[ 0, 0, \omega+\eta(\A, \infty) \right].
\end{equation}
\item Suppose that $\A_0$ and $\A_1$ are superconnections on
$\mathcal{H}$ such that
$\A_{0, [0]}- \A_{1, [0]} \in \Omega^0(M; op^0)$. Then
\begin{equation} \label{p3.4}
\left[ \mathcal{H}, \A_0, \omega \right]=\left[ \mathcal{H}, \A_1, \omega+\eta(\A_0, \A_1) \right].
\end{equation}
\end{enumerate}
\end{definition}

The map $\left[ \mathcal{H}, \A, \omega \right] \rightarrow
\Ch(\A) + d_H \omega$ passes to a map
$\check{K}^0_{\mathcal L}(M) \rightarrow
\Omega^{even}(M)$
whose image is $d_H$-closed.

\begin{proposition}
  Up to isomorphism, $\check{K}^0_{\mathcal L}(M)$ only depends
  on ${\mathcal L}$ through its stable isomorphism class;
  it is independent
  of the choices of connective structure
  and curving. 
\end{proposition}
\begin{proof}
Fixing the connective structure
  $\nabla$ on the gerbe ${\mathcal L}$, 
if $K'= K+\tau$ is another curving of $\nabla$ then the map
$\left[ \mathcal{H}, \A, \omega \right] \rightarrow \left[ \mathcal{H}, \A, e^{-\tau}\omega \right]$ induces an isomorphism of the corresponding
$\check{K}^0_{\mathcal L}(M)$-groups.

To see what happens when the connective structure on ${\mathcal L}$ varies,
suppose that 
$\nabla'=(\nabla_{\alpha \beta}')$ is another connective structure.
One can find a collection $\phi_\alpha \in \Omega^1(U_\alpha)$ such that 
$\nabla'_{\alpha \beta} -\nabla_{\alpha \beta} = \phi_\alpha -\phi_\beta$. 
If $K=(\kappa_\alpha)$ is a curving of $\nabla$ then we obtain a curving
$K'=(\kappa'_\alpha)$ of $\nabla^\prime$ by putting
$\kappa'_\alpha = \kappa_\alpha +d \phi_\alpha$.
The corresponding $3$-form $H' = d \kappa'_\alpha$ is the same as
$H = d\kappa_\alpha$. Suppose now that $\A$ is a superconnection on
an ${\mathcal L}$-projective Hilbert bundle, as defined using the
connective structure $\nabla$. Put
$\A_\alpha' = \A_\alpha -\phi_\alpha$. This defines a superconnection
$\A^\prime$ on 
$\mathcal{H}$ that is compatible with $\nabla'$.
Note that $\theta_{\A', K'}=\theta_{\A, K}$. It follows that the map 
$\left[ \mathcal{H}, \A, \omega \right] \rightarrow \left[ \mathcal{H}, \A', \omega \right]$  induces an isomorphism of the corresponding
$\check{K}^0_{\mathcal L}(M)$-groups. This isomorphism
depends on the choice of the $\phi_\alpha$'s.

To check what happens under stable isomorphism of the gerbe,
suppose that ${\mathcal L}^\prime$ is given by
${\mathcal L}_{\alpha \beta}^\prime =
\widehat{\mathcal L}_{\alpha} \otimes {\mathcal L}_{\alpha \beta}
\otimes \widehat{\mathcal L}_{\beta}^{-1}$.
Choose Hermitian connections $\widehat{\nabla}_\alpha$ on the
${\mathcal L}_\alpha$'s.
Put $\nabla^\prime_{\alpha \beta} = \widehat{\nabla}_\alpha
\otimes \nabla_{\alpha \beta} \otimes \widehat{\nabla}_\beta^{-1}$.
Given a cocycle
$\left[ \mathcal{H}, \A, \omega \right]$ for
$\check{K}^0_{\mathcal L}(M)$, we obtain a cocycle
$\left[ \mathcal{H}^\prime, \A^\prime, \omega^\prime \right]$ for
$\check{K}^0_{{\mathcal L}^\prime}(M)$ by
$\mathcal{H}^\prime_\alpha = \mathcal{H}_\alpha
\otimes \widehat{L}_\alpha^{-1}$,
$\A^\prime_\alpha = \A_\alpha \otimes \widehat{\nabla}_\alpha^{-1}$ and
$\omega^\prime = \omega$. This gives an isomorphism between
$\check{K}^0_{\mathcal L}(M)$ and $\check{K}^0_{{\mathcal L}^\prime}(M)$.
It depends on the choice of the $\widehat{\nabla}_\alpha$'s.
\end{proof}

\appendix

\section{Chern character in relative cohomology} \label{app}

In \cite{Quillen (1985)}, Quillen stated without proof that
if $\A$ is a superconnection on a finite dimensional vector
bundle $E$ over $M$, and the degree-$0$ component $\A_{[0]}$ is
invertible on an open subset $U \subset M$, then
the pair $(\Ch(\A), \eta(\A, \infty))$ represents the Chern
character of $E$ in relative cohomology.
In this appendix
we provide a proof of the statement in the
more general setting of superconnections on Hilbert bundles.
We give an application to a difference formula for eta forms.

\begin{remark}
As mentioned by Paradan and Vergne \cite{Paradan-Vergne (2009)},
in the finite dimensional case one can also prove Quillen's claim
by showing that the pair satisfies Schneiders' axiomatic
characterization of the relative Chern character
\cite[Proposition 4.5.2]{Schneiders (2000)}.
\end{remark}

Let $\A$ be a superconnection, in the sense of Definition \ref{def3},
on a $\Z_2$-graded Hilbert bundle ${\mathcal H}$ over a manifold $M$.
Let $U\subset M$ be an open set such that
$\A_{[0]}^2 \ge c^2 \Id >0$ on $U$,
for some $c > 0$.
Then $\eta(\A, \infty)$ is defined on $U$.

Recall that $\HH^*(M, U; \R)$ is the cohomology of the complex
\begin{equation} \label{A.1}
\Omega^*(M, U; \R) =
\Omega^*(M) \oplus \Omega^{*-1}(U)
\end{equation}
with differential
\begin{equation} \label{A.2}
d(\omega, \sigma) = (d \omega, i^* \omega - d \sigma),
\end{equation}
where $i : U \rightarrow M$ is the inclusion map.

\begin{lemma} \label{lem16} The pair $\left( \Ch (\mathbb{A}), \eta(\mathbb{A}, \infty)  \right)$ defines a class
in $\HH^*(M, U; \R)$
\end{lemma}
\begin{proof}
This follows from Lemmas \ref{lem7} and \ref{lem11}.
\end{proof}

\begin{lemma} \label{lem17} Let $\mathbb{A}$ and
$\mathbb{A}^\prime$ be two superconnections as above such that
$\A_{[0]} - \A_{[0]}^\prime \in \Omega^0(M; op^0({\mathcal H}))$ and
$\A_{[0]} = \A_{[0]}^\prime$ on $ U$. Then
\begin{equation} \label{A.3}
[\left( \Ch (\mathbb{A}), \eta(\mathbb{A}, \infty)  \right)]=[\left( \Ch (\mathbb{A}'), \eta(\mathbb{A}', \infty)  \right)] \in \HH^*(M, U; \R).
\end{equation}
\end{lemma}
\begin{proof} On the product space $[0,1] \times M$,
consider the superconnection
\begin{equation} \label{A.4}
\B = dt \wedge \partial_t +
t \mathbb{A} +(1-t) \mathbb{A}^\prime.
\end{equation}
Note that $\B_{[0]}^2 \ge c^2 \Id$
on $[0,1] \times U$.
Since $d \Ch(\B) = 0$ in $\Omega^*([0,1] \times M)$, one finds that
in $\Omega^*(M, U; \R)$, we have
\begin{align} \label{A.5}
& d \left( \int_0^1 \Ch (\B), \int_0^1 \eta(\B, \infty)  \right)
= \\
& \left( \Ch (\mathbb{A}), \eta(\mathbb{A}, \infty)  \right)-\left( \Ch (\mathbb{A}'), \eta(\mathbb{A}', \infty)  \right). \notag
\end{align}
The lemma follows.
\end{proof}

Since $\A_{[0]}$ is invertible on $U$, we can define a class
$\Ind(\A_{[0]}^\pm) \in \K^0(M, U)$ as follows.

Let $\gamma$ denote the $\Z_2$-grading operator on ${\mathcal H}$.
Put $\tH = \mathcal{H}\oplus \mathcal{H}$ with the $\Z_2$-grading operator
$\tg = \begin{bmatrix} \gamma& 0 \\
0& -\gamma \\
\end{bmatrix}$.
Let $f\in C^\infty(\mathbb{R})$ be a real-valued odd function such that
$xf(x) = 1 - \xi(x)$, with $\xi(0)=1$ and $\supp \xi(x) \subset [-c, c]$.
Put $Q= f(\A_{[0]})$.
Then $Q$ is a parametrix for $\A_{[0]}$
which is odd with respect to $\gamma$, self-adjoint and is the
inverse of $\A_{[0]}$ over $U$.

Put $S_0=1-Q \A_{[0]}$ and  $S_1=1-\A_{[0]}Q$.
Note that $S_0$ and $S_1$ vanish over $U$.
Put $L=\begin{bmatrix} S_0 &-(1+S_0)Q\\ \A_{[0]} &S_1 \end{bmatrix}$, with  $L^{-1}=\begin{bmatrix} S_0 &(1+S_0)Q\\-\A_{[0]} &S_1 \end{bmatrix}$. These are operators on $\tH$ that are even with respect to $\tg$.

Put $P=L^{-1}\begin{bmatrix}1 &0 \\0 &0 \end{bmatrix} L$ and
$P_0=\begin{bmatrix}0 &0 \\0 &1 \end{bmatrix}$. They are both
even projection operators. Then
\begin{equation}
  P-P_0 =
  \begin{pmatrix}
    S_0^2 & -S_0(1+S_0)Q \\
    -\A_{[0]} S_0 & - S_1^2
    \end{pmatrix}
  \end{equation}
is a smooth family of finite rank operators, i.e.
$P - P_0 \in \Omega^0(M; {\mathcal L}^{fr}(H))$; the smoothness
can be seen by repeated differentiation. 
Also, $P - P_0$ vanishes on $U$.

Put
$P^\pm = P|_{\tH^\pm}$ and
$P_0^\pm = P_0|_{\tH^\pm}$.

\begin{definition}
The index $\Ind \left( \A_{[0]}^\pm \right) \in K^0(M,U)$ is
represented by
the virtual projection
$[P^\pm - P_0^\pm]$.
\end{definition}

\begin{remark} \label{extrarem}
As in \cite[Chapter II.9.$\alpha$]{Connes (1994)},
$[P^\pm - P_0^\pm]$ is the index class in $K^0(M)$.
As $P^\pm - P_0^\pm$ vanishes on $U$, we can consider
$[P^\pm - P_0^\pm]$ to give a class in $K^0(M, U)$.
To justify this statement,
suppose that there is a closed subset $A \subset M$ which is a
strong deformation retract of $U$ (as is often the case).
Then there are  canonical isomorphisms
$K^0(M, U) \cong K^0(M, A) \cong \widetilde{K}^0(M/A)$.
The virtual projection $[P^\pm - P_0^\pm]$ descends to a
reduced $K$-theory class of $M/A$.
\end{remark}

Note that $\Ind \left( \A_{[0]}^- \right) =
 - \Ind \left( \A_{[0]}^+ \right)$.

The real-valued Chern character
$\Ch \left (\Ind \left( \A_{[0]}^+ \right) \right)$   in
$\HH^{even}(M, U; \R)$ can be represented as follows.
Put $\widetilde{\A} = \A \oplus \A$, a superconnection on
$\widetilde{H}$.
Then $\widetilde{\A}_{[1]}$ is a connection on $\widetilde{H}$.

\begin{definition}
The Chern character
of $\Ind \left( \A_{[0]}^+ \right)$ in $\widetilde{\HH}^{even}(M/A; \R)$ is
represented by
\begin{align} \label{A.7}
& \Ch \left (\Ind \left( \A_{[0]}^+ \right) \right) = \\
& \left[ \left( \frac12 \Tr_s \left( P e^{-(P\circ
\widetilde{\A}_{[1]} \circ P)^2} P
- P_0 e^{-(P_0 \circ \widetilde{\A}_{[1]}
\circ P_0)^2} P_0 \right), 0 \right) \right]. \notag
\end{align}
\end{definition}

\begin{remark}
To justify this definition, we first note that the closed form
\begin{equation} \label{nnew}
\frac12 \Tr_s \left( P e^{-(P\circ
\widetilde{\A}_{[1]} \circ P)^2} P
- P_0 e^{-(P_0 \circ \widetilde{\A}_{[1]}
\circ P_0)^2} P_0 \right)
\end{equation}
represents the Chern character of the index in $\HH^{even}(M; \R)$.
It has support in $M\setminus U$ and so,
in the setting of Remark \ref{extrarem}, extends
to a compactly supported
form in $(M/A) - (A/A)$. Under the isomorphisms
${\HH}^{even}_c((M/A) - (A/A); \R)
\cong \HH^{even}(M/A, A/A; \R) \cong
\HH^{even}(M, A; \R) \cong \HH^{even}(M, U; \R)$,
the form gets mapped to (\ref{A.7}).
\end{remark}

\begin{proposition} \label{prop9}
\begin{equation} \label{A.6}
\Ch \left (\Ind \left( \A_{[0]}^+ \right) \right) =
[\left( \Ch (\mathbb{A}), \eta(\mathbb{A}, \infty)  \right)].
\end{equation}
\end{proposition}
\begin{proof}

We claim that
\begin{align} \label{A.8}
& \left[ \left(
\Tr_s \left( P e^{-(P\circ \widetilde{\A}_{[1]} \circ P)^2} P -
P_0 e^{-(P_0 \circ \widetilde{\A}_{[1]} \circ P_0)^2} P_0 \right), 0
\right) \right] = \\
& \left[ \left( \Tr_s \left( P e^{-(P\circ \ta \circ P)^2} P -
P_0 e^{-(P_0 \circ \ta \circ P_0)^2} P_0 \right), 0 \right) \right] \notag
\end{align}
in $\HH^{even}(M, U; \R)$. To see this,
for $t \in [0,1]$ put
$\widetilde{\mathbb{A}}(t) = t \widetilde{\A}
+ (1-t) \widetilde{\A}_{[1]}$,
a superconnection on $\tH$.
Now
\begin{equation} \label{A.9}
P e^{-(P\circ \ta(t) \circ P)^2} P -
P_0 e^{-(P_0 \circ \ta(t) \circ P_0)^2} P_0
\end{equation}
is a $(2 \times 2)$-matrix
with entries in $\Omega^*(M; {\mathcal L}^{fr}({\mathcal H}))$, that are
smooth
in $t$.
Using the finite rank property, it is easy to justify that
\begin{align} \label{A.10}
& \frac{d}{dt}
\Tr_s \left( P e^{-(P\circ \ta(t) \circ P)^2} P -
P_0 e^{-(P_0 \circ \ta(t) \circ P_0)^2} P_0 \right) = \\
&
d \Tr_s \left( P \frac{d\ta(t)}{dt} P e^{-(P\circ \ta(t) \circ P)^2} P -
P_0 \frac{d\ta(t)}{dt} P_0e^{-(P_0 \circ \ta(t) \circ P_0)^2} P_0 \right),
\notag
\end{align}
so
\begin{align} \label{A.11}
& \Tr_s \left( P e^{-(P\circ \widetilde{\A}_{[1]} \circ P)^2} P -
P_0 e^{-(P_0 \circ \widetilde{\A}_{[1]} \circ P_0)^2} P_0 \right)
 = \\
& \Tr_s \left( P e^{-(P\circ \ta \circ P)^2} P -
P_0 e^{-(P_0 \circ \ta \circ P_0)^2} P_0 \right) + \notag \\
& d \int_0^1
\Tr_s \left( P \frac{d\ta(t)}{dt} P e^{-(P\circ \ta(t) \circ P)^2} P -
P_0 \frac{d\ta(t)}{dt} P_0e^{-(P_0 \circ \ta(t) \circ P_0)^2} P_0 \right)
\: dt. \notag
\end{align}
This proves (\ref{A.8}).

Considering
$P\circ \ta \circ P$ to be a superconnection on
$\Image(P)$, and $P_0 \circ \ta \circ P_0$ to be a superconnection on
$\Image(P_0)$,
Lemma \ref{lem16} implies that
\begin{equation} \label{A.12}
\left( \Tr_s  \left( e^{-(P\circ \ta \circ P)^2} \right),
\eta(P \circ \ta \circ P, \infty) \right)
\end{equation}
and
\begin{equation} \label{A.13}
\left( \Tr_s \left( e^{-(P_0\circ \ta \circ P_0)^2} \right),
\eta(P_0 \circ \ta \circ P_0, \infty) \right)
\end{equation}
are closed elements of $\Omega^{even}(M, U)$.
Since $P=P_0$ on $U$, we have
\begin{align} \label{A.14}
& \left[ \left( \Tr_s \left( P e^{-(P\circ \ta \circ P)^2} P -
P_0 e^{-(P_0 \circ \ta \circ P_0)^2} P_0 \right), 0 \right) \right] = \\
& \left[\left( \Tr_s  \left( e^{-(P\circ \ta \circ P)^2} \right),
\eta(P \circ \ta \circ P, \infty) \right) \right]- \notag \\
& \left[\left( \Tr_s \left( e^{-(P_0\circ \ta \circ P_0)^2} \right),
\eta(P_0 \circ \ta \circ P_0, \infty) \right) \right]. \notag
\end{align}

Next, since $P_0 = \begin{bmatrix}0 &0 \\0 &1 \end{bmatrix}$ and
the $\Z_2$-grading operator $\widetilde{\gamma}$ is $- \gamma$ on
$\Image(P_0)$, we have
\begin{equation} \label{A.15}
\left( \Tr_se^{-(P_0\circ \ta \circ P_0)^2} , \eta(P_0 \circ \ta \circ P_0, \infty) \right) =
( -\Ch (\mathbb{A}), -\eta(\mathbb{A}, \infty)).
\end{equation}

Finally, put
$P_1 = \begin{bmatrix}1 &0 \\0 &0 \end{bmatrix}$.
Then
\begin{equation} \label{A.16}
P \circ \ta \circ P =
L^{-1} P_1 L \circ \ta \circ L^{-1} P_1 L.
\end{equation}
Hence
\begin{align} \label{A.17}
& (\Ch(P\circ \ta \circ P), \eta(P \circ \ta \circ P, \infty))= \\
& (\Ch(P_1 L \circ \ta \circ L^{-1} P_1), \eta(P_1 L \circ \ta \circ L^{-1} P_1, \infty)). \notag
\end{align}
The superconnections $\mathbb{A}$ and $P_1 L \circ \ta \circ L^{-1} P_1$ on
$\Image(P_1) = \mathcal{H}$ satisfy the conditions of
Lemma \ref{lem17}. Hence
\begin{align} \label{A.18}
& [(\Ch(P_1 L \circ \ta \circ L^{-1} P_1), \eta(P_1 L \circ \ta \circ L^{-1} P_1, \infty))]= \\
& [\left( \Ch (\mathbb{A}), \eta(\mathbb{A}, \infty)  \right)]. \notag
\end{align}

The proposition now follows from combining
(\ref{A.7}), (\ref{A.8}), (\ref{A.14}), (\ref{A.15}), (\ref{A.17})
and (\ref{A.18}).
\end{proof}

\begin{corollary} \label{cor2} Suppose that $\{\E(t)\}_{t \in [0,1]}$
is a smooth $1$-parameter
family of superconnections on ${\mathcal H}$.
Put $\A = \E(0)$ and $\A^\prime = \E(1)$.
If $\A_{[0]}$ and $\A_{[0]}'$ are invertible then
\begin{equation} \label{A.19}
\eta(\A,\A'  ) - \eta(\A, \infty) + \eta(\A', \infty)
\in \Image ( \Ch : \K^1(M) \rightarrow \HH^{odd}(M; \R)).
\end{equation}
If $\E_{[0]}(t)$ is invertible for all $t \in [0,1]$
then the expression in (\ref{A.19}) vanishes in $\Omega^{odd}(M)/\Image(d)$.
\end{corollary}
\begin{proof}
Let $V$ be a codimension-zero submanifold of $M$ with compact closure.
(If $M$ is compact then we just take $V = M$.)
Put $M^\prime =
V \times \mathbb{R}$ and  $U= V \times ((-\infty,0)\cup (1, \infty))$.
There are isomorphisms
$i_K : \K^0(M^\prime, U) \rightarrow \K^1(V)$ and
$i_H  \: : \: \HH^{even}(M^\prime, U) \rightarrow \HH^{odd}(V)$, where $i_H$ is
represented by
\begin{equation} \label{A.20}
i_H (\omega, \eta) = \left( \int_0^1 \omega \right) - \eta(1) + \eta(0).
\end{equation}
The maps $i_K$ and $i_H$ are consistent in the sense that
\begin{equation} \label{A.21}
\Ch \circ i_K = i_H \circ \Ch.
\end{equation}

Extend $\E(t)$ to be constant on $(- \infty, 0)$ and constant on
$(1,\infty)$. By reparametrizing $\R$ if necessary, we can assume that
$\E$ is smooth in $t$.
Put $\B = dt \wedge \partial_t + \E(t)$, a superconnection on $M^\prime$.
The family $\B_{[0]}^+$ of operators defines an element
$[\B_{[0]}^+] \in \K^0(M^\prime, U)$,
since the operators are invertible on $U$.

Equations (\ref{2.9}) and (\ref{A.20})
along with Proposition \ref{prop9}, when applied to $\B$,
$M^\prime$ and $U$, give
\begin{equation} \label{A.22}
\Ch(i_K([\B_{[0]}^+])) = - \eta(\A, \A^\prime) - \eta(\A^\prime, \infty) +
\eta(\A, \infty).
\end{equation}
After exhausting $M$ by such $V$'s,
this proves the first part of the corollary.  If each $\E(t)$ is
invertible then $[\B_{[0]}^+]$ vanishes in $\K^0(M^\prime, U)$, which implies
the second part of the corollary.
\end{proof}

\end{document}